\documentclass[11pt, reqno]{amsart}
\usepackage{bbm}
\usepackage[top=2.7cm,left=2.8cm,right=2.8cm,bottom=2.7cm]{geometry}
\usepackage{amsfonts}
\usepackage{amsmath,amsthm,amscd,mathrsfs,amssymb,graphicx,enumerate,
	dsfont}
\usepackage{hyperref}
\usepackage[usenames,dvipsnames]{xcolor}
\usepackage{xypic}
\hypersetup{colorlinks=true,citecolor=NavyBlue,linkcolor=Brown,urlcolor=Orange}

\newtheorem{thm2}{Theorem}
\newtheorem{cor2}{Corollary}
\newtheorem{conj2}{Conjecture}

\newtheorem{ques2}{Question}
\newtheorem{thm}{Theorem}[section]

\newtheorem{lem}[thm]{Lemma}
\newtheorem{prop}[thm]{Proposition}
\newtheorem{conj}[thm]{Conjecture}
\theoremstyle{definition}
\newtheorem{defn}[thm]{Definition}

\newtheorem{ex}[thm]{Example}
\newtheorem{rmk}[thm]{Remark}

\newtheorem{ques}[thm]{Question}

 \DeclareMathOperator{\Spec}{Spec}

\DeclareMathOperator{\End}{End} 
\DeclareMathOperator{\Hom}{Hom}
\DeclareMathOperator{\Hilb}{Hilb}

 \DeclareMathOperator{\rk}{rk}
 \DeclareMathOperator{\pr}{pr} 

\DeclareMathOperator{\Pic}{Pic}

\newcommand{\C}{\ensuremath\mathds{C}}

\newcommand{\Z}{\ensuremath\mathds{Z}}
\newcommand{\Q}{\ensuremath\mathds{Q}}

\newcommand{\h}{\ensuremath\mathfrak{h}}

\newcommand{\D}{\ensuremath\mathrm{D}}
\newcommand{\Hdg}{\ensuremath\mathrm{Hdg}}
\newcommand{\HH}{\ensuremath\mathrm{H}}
\newcommand{\HHH}{\ensuremath\mathrm{HH}}
\newcommand{\CH}{\ensuremath\mathrm{CH}}

\newcommand{\id}{\ensuremath\mathrm{id}}
\newcommand{\pt}{\ensuremath\mathrm{pt}}
\newcommand{\tr}{\ensuremath\mathrm{tr}}

\newif\ifHideFoot
\HideFoottrue  % This line can be left alone.
\HideFootfalse  % Comment this line to hide footnotes

\ifHideFoot

\newcommand{\Lie}[1]{}
\newcommand{\Charles}[1]{}

\else 

\newcommand{\marg}[1]{\normalsize{{
			\color{red}\footnote{{\color{blue}#1}}}{\marginpar[\vskip
			-.25cm{\color{red}\hfill$\Rightarrow$\tiny\thefootnote}]{\vskip
				-.2cm{\color{red}$\Leftarrow$\tiny\thefootnote}}}}}
\newcommand{\Lie}[1]{\marg{(Lie) #1}}
\newcommand{\Charles}[1]{\marg{(Charles) #1}}

\fi

\begin{document}

	\title[A motivic global Torelli theorem	
	for isogenous K3 surfaces]{A motivic global
		Torelli theorem \linebreak 
		for isogenous K3 surfaces}
	
	\author{Lie Fu}
	\address{Institute for Mathematics, Astrophysics and Particle Physics, Radboud University, Nijmegen, Netherlands}
	\address{Institut Camille Jordan, Universit\'e Claude Bernard Lyon 1, France}
	\email{fu@math.univ-lyon1.fr}
	
	\author{Charles Vial}
	\address{Fakult\"at f\"ur Mathematik, Universit\"at Bielefeld, Germany} 
	\email{vial@math.uni-bielefeld.de}
	
	%\thanks{The first author is supported by ECOVA (ANR-15-CE40-0002), HodgeFun
	%(ANR-16-CE40-0011), LABEX MILYON (ANR-10-LABX-0070) of Universit\'e de Lyon
	%and
	%\emph{Projet Inter-Laboratoire} 2019 by F\'ed\'eration de Recherche en
	%Math\'ematiques Rh\^one-Alpes/Auvergne CNRS 3490.}
	
%	\date{\today}
	
	\begin{abstract} 
		We prove that the Chow motives of twisted derived equivalent K3 surfaces are
		isomorphic, not only as Chow motives (due to Huybrechts), but also as
		Frobenius algebra
		objects. Combined with a recent result of Huybrechts, we conclude that two
		complex projective K3
		surfaces are isogenous (i.e.\!~ their second rational cohomology groups are
		Hodge isometric)
		if and only if their Chow motives are isomorphic as Frobenius algebra
		objects\,; this can be regarded as a motivic Torelli-type theorem. 
		We ask whether, more generally, twisted derived equivalent hyper-K\"ahler
		varieties have isomorphic Chow motives as (Frobenius) algebra objects and in
		particular
		isomorphic graded rational cohomology algebras. In the appendix, we justify
		introducing the notion of ``Frobenius algebra object'' by showing the
		existence
		of an infinite family of K3 surfaces whose Chow motives are  pairwise
		non-isomorphic as Frobenius algebra objects but isomorphic as algebra objects.
		In particular, K3 surfaces in that family are pairwise non-isogenous but have
		isomorphic rational Hodge algebras. 
	\end{abstract}
	
	\maketitle

	\section*{Introduction}
	
	\subsection*{Torelli theorems for K3 surfaces}
	A compact complex surface is called a K3 surface if it is simply connected and
	has trivial canonical bundle. The Hodge structure carried by the second
	singular
	cohomology contains essential information of the surface. Indeed, the global
	Torelli theorem says that the isomorphism class of a K3 surface $S$ is
	determined by the Hodge structure $\HH^{2}(S, \Z)$ together with the
	intersection pairing on it \cite{PSS, BR}.
	
	The following more flexible notion due to Mukai \cite{Mukai} turns out to be
	crucial in the study of their derived categories\,: two complex projective K3
	surfaces $S$ and $S'$ are called \emph{isogenous} if
	there exists a \emph{Hodge isometry}
	$\varphi\colon \HH^2(S,\Q)\stackrel{\sim}{\longrightarrow} \HH^2(S',\Q)$,
	\emph{i.e.}\ an isomorphism of rational Hodge structures compatible with the
	intersection pairing on both sides.
	Recently, Buskin \cite{buskin} proved that such an isometry $\varphi$ is
	induced by an algebraic correspondence, as had previously been conjectured by
	Shafarevich \cite{Shaf} as a particularly interesting case of the Hodge
	conjecture. Let us call any such representative a \emph{Shafarevich cycle} for
	this isogeny. Shortly afterwards, Huybrechts
	\cite{huybrechts-isogenous} gave another proof and showed that in fact
	$\varphi$ is induced by a chain
	of exact linear equivalences between derived categories of twisted K3
	surfaces, thereby establishing the \emph{twisted derived
		global Torelli theorem} \cite[Corollary~1.4]{huybrechts-isogenous}\,:
	\begin{equation}\label{eq:TwDerTorelli}
	\text{$S$ and $S'$ are isogenous} \Longleftrightarrow \text{$S$ and $S'$ are
		twisted derived equivalent.}
	\end{equation}	
	Here, following Huybrechts \cite{huybrechts-isogenous}, we say that two K3
	surfaces $S$ and $S'$ are \emph{twisted derived equivalent} if there exist K3
	surfaces
	$S=S_0, S_1,\ldots, S_n=S'$ and Brauer classes $\alpha = \beta_0  \in
	\mathrm{Br}(S), \alpha_1 ,\beta_1 \in \mathrm{Br}(S_1),\ldots, \alpha_{n-1},
	\beta_{n-1} \in \mathrm{Br}(S_{n-1})$ and $\alpha' = \alpha_n \in
	\mathrm{Br}(S')$ and exact linear equivalences
	between bounded derived categories of twisted coherent sheaves
	\begin{equation}\label{eq:tde}
	\begin{array}{rlll}
	\mathrm{D}^b(S,\alpha)\simeq&\!\!\!\mathrm{D}^b(S_1,\alpha_1),&&\\
	&\!\!\!\mathrm{D}^b(S_1,\beta_1)\simeq&\!\!\!\mathrm{D}^b(S_2,\alpha_2) ,&\\
	&&\,\,\,\vdots&\\

	&&\!\!\!\mathrm{D}^b(S_{n-2},\beta_{n-2})\simeq&\!\!\!\mathrm{D}^b(S_{n-1},\alpha_{n-1}),\\
	&&&\!\!\!\mathrm{D}^b(S_{n-1},\beta_{n-1})\simeq\mathrm{D}^b(S',\alpha').
	\end{array}
	\end{equation}
	Note that by \cite{cs} any exact linear equivalence between bounded derived
	categories of twisted coherent sheaves on smooth projective varieties is of
	Fourier--Mukai type, so that in \eqref{eq:tde}, each equivalence
	$\mathrm{D}^b(S_i,\beta_i) \stackrel{\sim}{\longrightarrow}
	\mathrm{D}^b(S_{i+1},\alpha_{i+1})$ is induced by a Fourier--Mukai kernel
	$\mathcal{E}_i \in  \mathrm{D}^b(S_i\times S_{i+1},\beta_i^{-1} \boxtimes
	\alpha_{i+1})$ (unique up to isomorphism).
	\medskip
	
	Combined with his
	previous work \cite{huybrechts-derivedeq} generalized to the twisted case,
	Huybrechts  deduced that isogenous complex projective K3
	surfaces have isomorphic Chow motives\footnote{In this paper, Chow groups and
		motives are always with rational coefficients, except in Theorem
		\ref{thm:motglobtor}.}.
	However, the converse does not hold in general\,: there are K3 surfaces having
	isomorphic Chow motives (hence isomorphic rational Hodge structures) without
	being
	isogenous. Examples of such K3 surfaces were constructed geometrically in
	\cite{BSV}\,; see Remark~\ref{R:BSV} of Appendix~\ref{sect:Appendix}. We
	provide  in Theorem~\ref{thm:IsomChowMot}  two further constructions of
	infinite families of pairwise non-isogenous K3 surfaces with isomorphic Chow
	motives. 
	The motivation of the paper is to complete the picture by giving
	a motivic characterization of isogenous K3 surfaces (Corollary
	\ref{cor:torelli}). Our main result is the following\,:

	\begin{thm2}\label{thm:main}
		Let $S$ and $S'$ be two twisted derived equivalent K3 surfaces over a field
		$k$. 
		Then the Chow motives of $S$ and $S'$ are isomorphic as algebra objects, in
		fact even as Frobenius algebra objects (Definition \ref{def:FrobAlg}), in the
		category of rational Chow motives over $k$.
	\end{thm2}

	In concrete terms, there exists a correspondence $\Gamma \in \CH^{2}(S\times
	S')$ with $\Gamma \circ \delta_S =	\delta_{S'} \circ (\Gamma \otimes \Gamma)$
	(``algebra homomorphism'')  and such that $\Gamma$ is invertible as a morphism
	between the Chow motives of $S$ and $S'$ with $\Gamma^{-1}={}^{t}\Gamma$
	(``orthogonality''), where ${}^t\Gamma$ denotes the transpose of $\Gamma$ and
	where $\delta_S$ is the small diagonal in $S \times S\times S$ viewed as a
	correspondence between $S\times S$ and~$S$. Equivalently, $\Gamma$ is an
	isomorphism such that $(\Gamma\otimes\Gamma)_{*}\Delta_{S}=\Delta_{S'}$ and
	$(\Gamma\otimes\Gamma\otimes\Gamma)_{*}\delta_{S}=\delta_{S'}$.\linebreak
	The notion of Frobenius algebra object provides a conceptual way to pack these
	conditions. Roughly speaking, a Frobenius algebra object in a rigid tensor
	category is an algebra object together with an isomorphism to its dual
	object\footnote{Technically, one needs to further tensor the dual object by
		certain power of some tensor-invertible objects\,; this power is called the
		\emph{degree} of this Frobenius structure.} with some compatibility
	conditions.
	The Chow motive of a smooth projective variety carries a natural structure of
	Frobenius algebra object in the rigid tensor category of Chow motives. We refer
	to \S
	\ref{sec:Frob} for more details on Frobenius algebra objects. \medskip
	
	Note that a particular consequence of Theorem \ref{thm:main} is that 
	the induced action $\Gamma_* : \CH^*(S)\to \CH^*(S')$ on the Chow rings  is
	an isomorphism of graded $\Q$-algebras -- this can in fact be deduced from the
	previous work of Huybrechts \cite{huybrechts-derivedeq}\,; see
	Remark~\ref{rmk:ChowRingIso}. The latter, combined with Manin's identity principle, implies that the Chow motives of $S$ and $S'$ are isomorphic. By de Cataldo--Migliorini \cite{dCM}, one deduces that the Chow motives of the Hilbert schemes of length-$n$ subschemes $\Hilb^{n}(S)$ and $\Hilb^{n}(S')$ are isomorphic. However, having an isomorphism of Frobenius
	algebra objects allows us to derive the following much stronger result\,:
	
	\begin{cor2}[Powers and Hilbert schemes]\label{cor:Powers}
		Let $S$ and $S'$ be two twisted derived equivalent K3 surfaces defined over a
		field $k$. Then for any positive integers $n_{1}, \ldots, n_{r}$, there is an
		isomorphism of Frobenius algebra objects $$\h\left(\Hilb^{n_{1}}(S)\times
		\cdots
		\times \Hilb^{n_{r}}(S)\right)\simeq \h\left(\Hilb^{n_{1}}(S')\times \cdots
		\times \Hilb^{n_{r}}(S')\right).$$ As a consequence, there is an algebraic
		correspondence inducing an isomorphism of \emph{graded} $\Q$-algebras\,:
		$$\CH^{*}\left(\Hilb^{n_{1}}(S)\times \cdots \times
		\Hilb^{n_{r}}(S)\right)\simeq \CH^{*}\left(\Hilb^{n_{1}}(S')\times \cdots
		\times
		\Hilb^{n_{r}}(S')\right).$$
		Here $\Hilb^{n}(S)$ denotes the Hilbert scheme of length-$n$ subschemes on
		$S$.
	\end{cor2}

	\medskip

	Now let the base field be the field of complex numbers $\C$. Combining Theorem
	\ref{thm:main} with Huybrechts' twisted derived
	global Torelli theorem \eqref{eq:TwDerTorelli} mentioned above, we can
	establish\,:
	
	\begin{cor2}[Motivic global Torelli theorem for isogenous K3 surfaces]
		\label{cor:torelli}
		Let $S$ and $S'$ be two complex projective K3 surfaces. The following
		statements are
		equivalent\,:
		\begin{enumerate}[(i)]
			\item $S$ and $S'$ are isogenous\,;
			\item $S$ and $S'$ are twisted derived equivalent\,;
			\item $\h(S)$ and $\h(S')$ are isomorphic as Frobenius algebra objects.
		\end{enumerate}	
	\end{cor2}
	
	In the appendix, we construct in Theorem \ref{thm:IsomChowMot} an infinite
	family of pairwise non-isogenous K3 surfaces whose motives are all isomorphic
	as
	algebra objects. This justifies introducing the Frobenius structure. In
	addition, Proposition~\ref{prop:IsoChowRing3} gives some evidence that the
	isogeny class of a K3 surface cannot be determined solely by its Chow ring.\medskip

	Finally, with integral coefficients, an algebra isomorphism between the motives
	of two K3 surfaces must respect the Frobenius structure. Therefore, the
	classical global Torelli theorem \cite{PSS} can be upgraded to a \emph{motivic
		global
		Torelli
		theorem}\,:
	\begin{equation*}\label{eq:MotGlobTorelli}
	\text{$S$ and $S'$ are isomorphic} \Longleftrightarrow \text{their integral
		Chow
		motives are isomorphic as algebra objects.}
	\end{equation*}	
	We refer to Theorem~\ref{thm:motglobtor} for a proof.

	\subsection*{Orlov conjecture and multiplicative structure}
	The proof of Theorem~\ref{thm:main} relies on the Beauville--Voisin
	decomposition of the small diagonal of a K3 surface (see Theorem
	\ref{thm:bv})\,: given an exact linear equivalence  $
	\operatorname{D}^b(S,\alpha)
	\stackrel{\sim}{\longrightarrow}  \operatorname{D}^b(S',\alpha')$  between
	twisted K3 surfaces, we are reduced to exhibiting a correspondence $\Gamma \in
	\CH^2(S\times S')$ such that $\Gamma \circ {}^t\Gamma = \Delta_{S'}$,
	${}^t\Gamma \circ \Gamma = \Delta_{S}$ and $\Gamma_*o_{S} = o_{S'}$, where
	$o_{S}=\frac{1}{24}c_{2}(S)$ is the Beauville--Voisin 0-cycle \cite{bv}. The
	key
	point then consists in 
	showing that if $v_2(\mathcal{E})$ denotes the dimension-2 component of the
	Mukai
	vector of the Fourier--Mukai kernel of an exact linear equivalence 
	$\Phi_{\mathcal{E}} : \operatorname{D}^b(S,\alpha)
	\stackrel{\sim}{\longrightarrow}  \operatorname{D}^b(S',\alpha')$  between
	twisted smooth projective surfaces, then $v_2(\mathcal{E})$ induces an
	isomorphism $\h^2_{\mathrm{tr}}(S) 	\stackrel{\sim}{\longrightarrow} 
	\h^2_{\mathrm{tr}}(S')$ of the transcendental motives of $S$ and $S'$ with
	inverse given by $v_2(\mathcal{E}^\vee\otimes p^* \omega_S)$, where
	$\mathcal{E}^\vee$
	denotes the derived dual of $\mathcal{E}$ and $p : S\times S' \to S$ is the
	natural projection. (In the case of K3 surfaces, we have
	$v_2(\mathcal{E}^\vee\otimes p^* \omega_S)= {}^tv_2(\mathcal{E})$.) This is
	achieved by
	exploiting known cases of Murre's Conjecture~\ref{conj:murre}\eqref{B}, and we
	thereby give an alternative proof of Huybrechts' \cite[Theorem
	0.1]{huybrechts-derivedeq}, generalized to all surfaces\,: \emph{two twisted
		derived equivalent smooth projective
		surfaces have isomorphic Chow motives}\,; see Theorem \ref{thm:huybrechts}. 
	This confirms the two-dimensional case of the following conjecture due to
	Orlov\,:
	\begin{conj2}[Orlov \cite{orlov}] \label{conj:orlov}
		Let $X$ and $Y$ be two derived equivalent smooth projective varieties. Then
		their Chow motives are isomorphic.
	\end{conj2}
	We illustrate also in \S \ref{subsec:3-4} how the same techniques can be used
	to establish
	Orlov's Conjecture~\ref{conj:orlov} in some new cases in dimension 3 and 4\,;
	see
	Proposition \ref{prop:generalization}. 
	\medskip
	
	In view of Theorem \ref{thm:main}, we naturally ask under what
	circumstances one could expect a ``multiplicative Orlov conjecture'', namely
	whether
	two
	derived equivalent smooth projective varieties have isomorphic Chow motives as
	algebra objects, or even as Frobenius algebra objects.  According to the
	celebrated theorem of Bondal--Orlov
	\cite{bo}, this holds true for varieties with ample or anti-ample canonical
	bundle, since any two such derived equivalent varieties must be isomorphic.
	The situation gets more intriguing for varieties with trivial canonical bundle
	and we cannot expect in general that derived equivalent varieties
	have isomorphic Chow motives as Frobenius algebra objects\,: there exists 
	counter-examples for Calabi--Yau threefolds and abelian varieties, where even
	the graded cohomology
	algebras of the two derived equivalent varieties are not isomorphic as
	Frobenius algebras (see
	Example \ref{ex:BC} and Proposition \ref{prop:AV} $(ii)$).
	We notice that, on the other hand, two derived equivalent abelian varieties are
	isogenous and have isomorphic Chow motives as algebra objects (see Proposition
	\ref{prop:AV} $(i)$).
	
	Although we do not provide much evidence beyond the
	case of K3 surfaces, we are tempted to ask, 
	because of the (expected) similarities of the intersection product on
	hyper-K\"ahler varieties with that on abelian varieties 
	(\emph{cf.} Beauville's
	seminal \cite{beauville-splitting}, and also \cite{sv, fv}), the following
	
	\begin{ques2}\label{conj:main}
		Let $X$ and $Y$ be two twisted
		derived equivalent projective hyper-K\"ahler varieties. Are their Chow motives
		isomorphic
		as algebra objects or even as Frobenius algebra objects\,? In particular, are
		their cohomology 
		$\HH^*(-,\Q)$ isomorphic as graded $\Q$-algebras or even as Frobenius
		algebras\,?
	\end{ques2}
	Corollary \ref{cor:Powers} gives an example in higher dimensions. See \S
	\ref{sect:MultOrlov} for other examples, conjectures and rudiment
	discussions on this subject.
	In our accompanying work \cite{FVcubic}, we answer Question~\ref{conj:main} in a non-commutative setting. Namely, we show that two cubic fourfolds with Fourier--Mukai equivalent Kuznetsov components (which can be considered as ``non-commutative K3 surfaces") have isomorphic rational Chow motives as Frobenius algebra objects.

	\subsection*{Canonicity of the Shafarevich cycle}
	
	In \cite{huybrechts-isogenous}, Huybrechts shows that the restriction to the
	transcendental
	cohomology of an isogeny 	$\varphi\colon
	\HH^2(S,\Q)\stackrel{\sim}{\longrightarrow} \HH^2(S',\Q)$ 
	is induced by the cycle $v_2(\mathcal{E}_{n-1}) \circ \cdots \circ
	v_2(\mathcal{E}_{0}) \in \CH^2(S\times S')$, 
	where $\mathcal{E}_0,\ldots,\mathcal{E}_{n-1}$ are the Fourier--Mukai kernels
	in~\eqref{eq:tde}. This provides a Shafarevich cycle for the isogeny $\varphi$.
	In \S \ref{sec:v2}, 
	we  give some evidence for the above cycle to be canonical, that is,
	independent of the choice 
	of a chain of twisted derived equivalence inducing the isogeny. This depends on
	extending a result of Huybrechts 
	and Voisin (Theorem~\ref{thm:hv}) to twisted equivalences. We do however prove
	unconditionally 
	in Theorem~\ref{thm:product-c2} that the intersection of the second Chern
	classes of two 
	objects $\mathcal{E}_1$ and $\mathcal{E}_2$ in $\mathrm{D}^b(S\times S')$
	inducing 
	an equivalence $\mathrm{D}^b(S) \stackrel{\sim}{\longrightarrow}
	\mathrm{D}^b(S')$ 
	is proportional to $c_2(S)\times c_2(S')$ in $\CH^2(S\times S')$. This suggests
	that the Mukai vectors of twisted derived equivalences between K3 surfaces can
	be added to the Beauville--Voisin ring\,; see \S \ref{sec:v2}. \medskip

	\subsection*{Notation and Conventions} We fix a base field $k$. By a
	derived equivalence between smooth projective $k$-varieties, we mean a
	$k$-linear exact equivalence of triangulated categories between their bounded
	derived categories of 
	coherent sheaves. Chow groups will always be considered with
	rational coefficients. Concerning the category of Chow motives over $k$, we
	follow the notation and conventions of \cite{andre}. This category is a
	pseudo-abelian rigid tensor category, whose objects consist of triples
	$(X,p,n)$,
	where $X$ is a smooth projective variety of dimension $d_X$ over $k$, $p\in
	\CH^{d_X}(X\times_k X)$ with $p\circ p = p$, and $n\in \Z$. Morphisms $f:
	M=(X,p,n) \to N=(Y,q,m)$ are elements $\gamma \in \CH^{d_X+m-n}(X\times_k Y)$
	such that $\gamma \circ p = q\circ \gamma = \gamma$. The tensor product of two
	motives
	is defined in the obvious way, while the dual of $M=(X, p, n)$ is $M^\vee =
	(X,{}^tp,-n+d_X)$, where ${}^tp$ denotes the transpose of $p$. The Chow motive
	of a smooth projective variety $X$ is defined as $\h(X) := (X,\Delta_X,0)$,
	where $\Delta_X$ denotes the class of the diagonal inside $X\times X$, and the
	\emph{unit motive} is denoted $\mathds{1} := \h(\Spec (k))$. In
	particular, we have $\CH^l(X) = \Hom(\mathds{1}(-l),\h(X))$. The
	\emph{Tate
		motive} of weight $-2i$ is the motive $\mathds{1}(i) := (\Spec (k),
	\Delta_{\Spec(k)},
	i)$. A motive is said to be of \emph{Tate type} if it is isomorphic to a
	direct sum of Tate motives.

	\subsection*{Acknowledgments} 
	We thank 
	Benjamin Bakker, Chiara Camere and Jean-Yves Welschinger
	for helpful discussions on the Appendix. We also thank the referee for useful suggestions.
	This paper was completed at the Institut Camille Jordan in Lyon where the
	second
	author's stay was  supported by the CNRS. 
	The first author is supported by the Radboud  Excellence Initiative programme, ECOVA
	(ANR-15-CE40-0002), HodgeFun (ANR-16-CE40-0011), LABEX MILYON
	(ANR-10-LABX-0070)
	of Universit\'e de Lyon and \emph{Projet Inter-Laboratoire} 2019 by
	F\'ed\'eration de Recherche en Math\'ematiques Rh\^one-Alpes/Auvergne CNRS
	3490.

	\section{Derived equivalent surfaces}
	
	The aim of this section is to provide an alternative proof
	to the following result of Huybrechts~\cite{huybrechts-derivedeq,
		huybrechts-isogenous}\,: 
	
	\begin{thm}[Huybrechts] \label{thm:huybrechts}
		Let $S$ and $S'$ be two (twisted)
		derived equivalent smooth projective surfaces
		defined over a field $k$. Then $S$ and $S'$ have isomorphic Chow motives.
	\end{thm}
	
	The reason for including such a proof is threefold\,: first it provides anyway
	all
	the prerequisites and notation for the proof of Theorem \ref{thm:main} which
	will be given in \S\ref{S:mainthm}\,;
	second by avoiding Manin's identity principle\footnote{Manin's identity
		principle only establishes that $ v_2(\mathcal{E}^\vee\otimes
		p^*\omega_S)\circ
		v_2(\mathcal{E})$ acts as the identity on $\CH^2(\h^2_{\mathrm{tr}}(S))$ which
		implies that it is unipotent as an endomorphism of $\h^2_{\mathrm{tr}}(S)$.}
	(as is employed in \cite{huybrechts-derivedeq})
	we obtain an explicit inverse to
	the isomorphism $v_2(\mathcal{E}) : \h^2_{\mathrm{tr}}(S)
	\stackrel{\sim}{\longrightarrow} \h^2_{\mathrm{tr}}(S')$ which will be
	essential to the
	proof of Theorem \ref{thm:main}\,;
	and third it provides somehow a link between
	Orlov's Conjecture \ref{conj:orlov} and Murre's Conjecture \ref{conj:murre}
	which itself is intricately linked to the conjectures of Bloch and Beilinson
	(see \cite{jannsen}).

	\subsection{Murre's conjectures}
	
	We fix a base field $k$ and a Weil cohomology theory $\HH^*(-)$ for smooth
	projective
	varieties over $k$. Concretely, we think of $\HH^*(-)$ as Betti cohomology in
	case $k\subseteq \C$, or as $\ell$-adic cohomology when $\mathrm{char}(k) \neq
	\ell$.
	
	\begin{conj}[Murre \cite{murre}]\label{conj:murre}
		Let $X$ be a smooth projective variety of dimension $d$ over~$k$. 
		\begin{enumerate}[(A)]
			\item \label{A} The Chow motive $\h(X)$ has a Chow--K\"unneth decomposition
			(also
			called \emph{weight decomposition})  $\h(X) =
			\h^0(X) \oplus \cdots \oplus \h^{2d}(X)$, meaning that $\HH^*(\h^i(X)) =
			\HH^i(X)$ for all $i$.
			\item \label{B} $\CH^l(\h^i(X)) := \Hom(\mathds{1}(-l),\h^i(X)) = 0$ for all
			$i>2l$ and for all $i<l$.
			\item \label{C} The filtration $\mathrm{F}^k\CH^l(X) := \CH^l\big( 
			\bigoplus_{i\leq 2l-k} \h^i(X) \big)$ does not depend on the choice of a
			Chow--K\"unneth
			decomposition.
			\item \label{D} $\mathrm{F}^1\CH^l(X) = \CH^l(X)_{\mathrm{hom}} := \ker
			(\CH^l(X) \xrightarrow{\operatorname{cl}}
			\HH^{2l}(X))$.
		\end{enumerate}
	\end{conj}
	
	The filtration defined in \eqref{C} is conjecturally the Bloch--Beilinson
	filtration \cite{Beilinson} (see also \cite[Chapter 11]{VoisinBook}).
	In fact, as shown by Jannsen \cite{jannsen}, the conjecture of Murre holds for
	all smooth projective varieties if and only if the conjectures of
	Bloch--Beilinson hold. We refer to \cite{jannsen} for precise statements. 
	
	\begin{prop}\label{prop:formal}
		Let $X$ and $Y$ be smooth projective varieties over $k$. Assume that $\h(X)$
		and	$\h(Y)$ admit a Chow--K\"unneth decomposition as in
		Conjecture~\ref{conj:murre}\eqref{A}.
		Then, with respect
		to the Chow--K\"unneth decomposition $$\h^n(X\times Y) = \bigoplus_{i+j = n}
		\h^{2d_X-i}(X)^\vee(-d_X)\otimes \h^j(Y),$$
		the product $X\times Y$ satisfies
		Conjecture~\ref{conj:murre} \eqref{B}  if and only if 
		\begin{equation}\label{eq}
		\Hom(\h^i(X),\h^j(Y)(k))=  0\quad  \text{for all } i<j-2k \text{ and for all
		}
		i>j+d_X-k.
		\end{equation}
	\end{prop}
	\begin{proof}
		This is formal\,: we have 
		\begin{align*}
		\Hom(\mathds{1}(-k-d_X), \h^{n}(X\times Y)) 
		& = \bigoplus_{i-j=2d_{X}-n}	\Hom(\mathds{1}(-k), \h^j(Y)\otimes
		\h^i(X)^\vee)\\
		& = \bigoplus_{i-j=2d_{X}-n}\Hom(\h^i(X),\h^j(Y)(k)).\qedhere
		\end{align*}
	\end{proof}

	In other words, Murre's conjecture \eqref{B} implies that
	a motive of pure weight does not admit any non-trivial morphism to a motive of
	pure larger weight.

	\begin{thm}[Murre \cite{murre-surfaces}] \label{thm:product}
		Let $X$ be a smooth projective irreducible variety of dimension $d_X$ over a
		field~$k$.
		\begin{enumerate}[(i)]
			\item \label{i} The Chow motive of $X$ admits a decomposition $$\h(X) =
			\h^0(X) \oplus \h^1(X) \oplus M \oplus \h^{2d_X-1}(X) \oplus \h^{2d_X}(X)$$
			such
			that
			$\HH^*(\h^i(X)) = \HH^i(X)$\,; in particular, Conjecture
			\ref{conj:murre}\eqref{A} holds for curves and
			surfaces. Moreover, such a decomposition can be chosen such that
			\begin{itemize}
				\item   $\h^{2d_X}(X)(d_X) = \h^0(X)^\vee \simeq \h^0(X)$ and
				$\h^{2{d_X}-1}(X)(d_X) = \h^1(X)^\vee \simeq \h^1(X)(1)$\,;
				\item 	$\h^0(X)$
				is the unit motive $\mathds 1$ and   $\h^1(X) \simeq
				\h^1(\Pic^0(X)_{\mathrm{red}})$\,;
				\item $\Hom(\mathds{1}(-i),\h^1(X)) = 0$ for
				$i\neq
				1$, and  $\Hom(\mathds{1}(-1),\h^1(X)) = \Pic^0(X)_{\mathrm{red}}(k)\otimes
				\Q$.
			\end{itemize}
			\item \label{ii}Equation \eqref{eq} holds in case $X$ and $Y$ are varieties
			of
			dimension $\leq 2$ endowed with a Chow--K\"unneth decomposition as in
			\eqref{i}.
		\end{enumerate}
	\end{thm}
	\begin{proof} Item \eqref{i} in the case of surfaces is the main result of
		\cite{murre-surfaces}. In fact, for any smooth projective variety $X$ of any
		dimension, $\h^1(X)$ can be constructed as a direct summand of the motive of
		a
		smooth projective curve.
		As for \eqref{ii}, this was checked by Murre \cite{murre2} in the case one  of
		$X$ and $Y$ has dimension $\leq 1$. Thanks to item \eqref{i} and Proposition
		\ref{prop:formal}, it only remains
		to
		check that $\CH^l(\h^2(X) \otimes \h^2(Y)) = 0$ for $l = 0,1$ for smooth
		projective surfaces $X$ and $Y$. For that purpose, we simply observe that for
		any choice of a Chow--K\"unneth decomposition (if it exists) on the motive of
		a
		smooth
		projective variety $Z$ we have $\CH^0(Z) = \CH^0(\h^0(Z))$ and $\CH^1(Z) =
		\CH^1(\h^2(Z) \oplus \h^1(Z))$. (Indeed, denote $\pi^i_Z$ the projectors
		corresponding to the Chow--K\"unneth decomposition of $Z$, then by definition
		$\pi^{2i}_Z$ acts as the identity on $\HH^{2i}(Z)$ and hence on
		$\mathrm{im}(\CH^i(Z) \to \HH^{2i}(Z))$, and by Murre \cite{murre-surfaces}
		$\pi^1_Z$ acts as the identity on $\mathrm{ker}(\CH^1(Z) \to \HH^{2}(Z))$.
		Therefore $\pi^2_Z
		+
		\pi^1_Z$, which is a projector, acts as the identity on $\CH^1(Z)$.)
	\end{proof}

	The following terminology will be convenient for our purpose. We say that a
	Chow motive $M$ is
	of \emph{curve type} (or of \emph{pure weight 1}) if it is isomorphic
	to a direct summand of a direct sum of motives of the form $\h^1(C)$, where $C$
	is a smooth projective curve defined over $k$.
	Motives of curve type form a thick additive subcategory\footnote{It is however
		not stable under tensor products and it does not contain any Tate motives. The
		tensor category generated by the Tate motives and the motives of curve type
		consists of the so-called motives of \emph{abelian type}.} and enjoy the
	following property, which is also shared by Tate	motives\,:

	\begin{prop}\label{prop:curve}
		The full subcategory of motives whose objects are of curve type is abelian
		semi-simple. Moreover, the realization functor $M \mapsto \HH^*(M)$ is
		conservative\footnote{A functor $F:\mathcal C \to \mathcal D$ is said to be \emph{conservative} if it is ``isomorphism-reflecting'', that is, if $f:C\to C'$ is a morphism in $\mathcal C$ such that $F(f)$ is an isomorphism in $\mathcal D$, then $f$ is an isomorphism in $\mathcal C$.}.
	\end{prop}
	\begin{proof}
		The first statement follows from the fact that this full subcategory of
		motives of curve type is equivalent to the category of abelian varieties up to
		isogeny, via the Jacobian construction\,; see
		\cite[Proposition~4.3.4.1]{andre}.
		The second
		statement follows from the first one together with the fact that
		$\HH^{*}(\h^{1}(A))$ is a $2g$-dimensional vector space for an abelian variety
		$A$ of  dimension~$g$.
	\end{proof}

	\subsection{Proof of Theorem \ref{thm:huybrechts}}
	First, recall that, given a twisted K3 surface $(S, \alpha)$, that is, a K3 surface $S$ equipped with a Brauer class $\alpha\in \operatorname{Br}(S)$, 
	 the Chern character of an $\alpha$-twisted sheaf~$E$ is defined as follows\,: choose a positive integer $n$ such that $\alpha^n=1$ (hence $E^{\otimes n}$ is untwisted), then
	\[\operatorname{ch}(E):=\sqrt[n]{\operatorname{ch}(E^{\otimes n})}.\]
	The definition is independent of the choice of $n$, can be extended naturally to the whole derived category of twisted sheaves $\D^b(S, \alpha)$, and satisfies the usual compatibilities with tensor operations (see \cite[Section 2.1]{huybrechts-isogenous}). We then define the \emph{Mukai vector} $v(E):=\operatorname{ch}(E)\sqrt{\operatorname{td}(T_S)}$.
	
	We can thus observe as in \cite[Section 2]{huybrechts-isogenous} that for twisted
	equivalences the yoga of Fourier--Mukai kernels, their action on Chow groups
	induced by Mukai vectors, and how they behave under convolutions works as in
	the
	untwisted case, as long as we work with rational coefficients. Therefore, for ease of notation, we will only give a proof of
	Theorem \ref{thm:huybrechts} in the untwisted case.
	
	\subsubsection{Derived equivalences and motives, following
		Orlov.}\label{subsec:orlov}
	In general, let $\Phi_{\mathcal{E}} : \mathrm{D}^b(X)
	\stackrel{\sim}{\longrightarrow} \mathrm{D}^b(Y)$ be an exact equivalence with
	Fourier--Mukai kernel $\mathcal{E} \in \mathrm{D}^b(X\times Y)$ between the
	derived categories of two
	smooth projective $k$-varieties of dimension $d$. 
	Its inverse can be described as $\Phi_{\mathcal{E}}^{-1} \simeq
	\Phi_{\mathcal{E}^\vee\otimes p^*\omega_X[d]}\simeq
	\Phi_{\mathcal{E}^\vee\otimes q^*\omega_Y[d]}$, where $\mathcal{E}^\vee$ is the
	derived dual of $\mathcal{E}$ and $p, q$ are the projections from $X\times Y$
	to $X$ and $Y$ respectively.
	As observed by Orlov \cite{orlov},
	the Mukai vector 
	$$v(\mathcal{E}) := \mathrm{ch}(\mathcal{E})\cdot \sqrt{\mathrm{td}(X\times
		Y)} \in \CH^{*}(X\times Y)$$ induces a split injective morphism 
	of motives $\h(X) \longrightarrow \bigoplus_{i=-d}^{d} \h(Y)(i)$ with left
	inverse given by $v(\mathcal{E}^\vee\otimes p^*\omega_X[d])$, \emph{i.e.}
	$$\xymatrix{\mathrm{id} : \ \h(X) \ar[rrr]^{v(\mathcal{E})\quad } &&&
		\bigoplus_{i=-d}^{d} \h(Y)(i) \ar[rrr]^{\qquad v(\mathcal{E}^*\otimes
			p^*\omega_X[d])} && & \h(X).
	}$$
	In particular, $v(\mathcal{E}^\vee\otimes p^*\omega_X[d]) \circ v(\mathcal{E})
	=
	\Delta_{X}$. In fact, as observed by Orlov \cite{orlov}, the latter identity
	shows that $v(\mathcal E)$ seen as a morphism of ind-motives $	\bigoplus_{i\in
		\Z} \h(X)(i) \to 	\bigoplus_{i\in \Z} \h(Y)(i)$ is an isomorphism with inverse
	given by $v(\mathcal{E}^\vee\otimes p^*\omega_X[d])$.

	\subsubsection{The refined decomposition of the motive of surfaces, following
		Kahn--Murre--Pedrini}
	\label{subsec:kmp}
	Let $S$ be a smooth projective surface over $k$. The motive $\h(S)$ admits a
	Chow--K\"unneth decomposition as in Murre's Theorem
	\ref{thm:product}\eqref{i}\,; in
	particular $\h^0(S) = \h^4(S)(2)^\vee$ is the unit motive $\mathds 1$ and
	$\h^1(S) =
	\h^3(S)(2)^\vee$ is of curve type. Following
	Kahn--Murre--Pedrini \cite{kmp}, the summand $\h^2(S)$ admits a further
	decomposition $$\h^2(S) = \h^2_{\mathrm{alg}}(S) \oplus \h^2_{\mathrm{tr}}(S)$$
	defined as
	follows. Let $k^s$ be a separable closure of $k$ and let $E_1,\ldots,E_\rho$ be
	non-isotropic
	divisors in $\CH^1(S_{k^s})$ whose images in
	$\CH^1(\h^2(S_{k^s})) = \mathrm{NS}(X_{k^s})_\Q$ form an orthogonal basis. Up to replacing each $E_i$
	by $(\pi^2_S)_*E_i$, we can assume that $E_i$ belongs to $\CH^1(\h^2(S_{k^s}))$ for all $i$.
	Consider then the idempotent correspondence
	$$\pi^2_{\mathrm{alg},S} := \sum_{i=1}^\rho \frac{1}{\deg (E_i\cdot E_i)}E_i
	\times
	E_i.$$ 
	Since $\pi^2_{\mathrm{alg},S}$ is the intersection form on
	$\CH^1(\h^2(S_{k^s})) = \mathrm{NS}(S_{k^s})_\Q$,
	it is Galois-invariant, and hence
	 does define an idempotent in $\CH^2(S\times
	S)$. The motive $(S,\pi^2_{\mathrm{alg},S},0)$ is clearly isomorphic, after
	base-change to $k^s$, to the direct sum of $\rho$ copies of the Tate
	motive
	$\mathds{1}(-1)$.  Moreover, it is easy to check that $\pi^2_{\mathrm{alg},S} $
	is orthogonal to
	$\pi^i_S$ for $i\neq 2$ (use $(\pi^1_S)_*E_i = 0$). 
	Equivalently, we have $\pi^2_{\mathrm{alg},S} \circ \pi_S^2 = \pi^2_S \circ
	\pi^2_{\mathrm{alg},S} $, so that $(S,\pi^2_{\mathrm{alg},S},0)$ does define a
	direct summand of
	$\h^2(S)$, denoted by $\h^{2}_{\mathrm{alg}}(S)$. We then define
	$\pi^2_{\mathrm{tr},S} := \pi^2_S -
	\pi^2_{\mathrm{alg},S} $ and $\h^{2}_{\mathrm{tr}}(S):=(S,
	\pi^2_{\mathrm{tr},S}, 0)$. 
	It is then straightforward to check that such a decomposition satisfies
	$\Hom(\mathds{1}(-i),\h^2_{\mathrm{tr}}(S)) = 0$ for all $i\neq 2$. 
	We note that ${}^t	\pi^2_{\mathrm{alg},S} = 	\pi^2_{\mathrm{alg},S} $, and
	since ${}^t	\pi_S^0 = \pi_S^4$ and ${}^t	\pi_S^1 = \pi_S^3$ we also have ${}^t
	\pi^2_{\mathrm{tr},S} = 	\pi^2_{\mathrm{tr},S}$.
	Moreover, although
	this won't be of any use to us, we mention for comparison to
	\cite{huybrechts-derivedeq} that
	$\Hom(\mathds{1}(-2),\h^2_{\mathrm{tr}}(S)) $ coincides with the Albanese
	kernel.
	
	More generally, the above refined decomposition can be performed for direct
	summand of motives of surfaces, \emph{i.e.} for motives of the form $(S,p,0)$,
	where $S$ is a smooth projective $k$-surface and $p\in \CH^2(S\times S)$ is an
	idempotent. This will be used in the proof of
	Proposition~\ref{prop:generalization}.
	Indeed, by the above together with Theorem \ref{thm:product}\eqref{ii}, we have
	a decomposition
	\begin{equation}\label{eq:dec}
	\h(S) = \h^0(S) \oplus  \h^1(S) \oplus \h^2_{\mathrm{alg}}(S) \oplus
	\h^2_{\mathrm{tr}}(S) \oplus \h^3(S) \oplus \h^4(S),
	\end{equation}
	where none of the direct summands admit a non-trivial morphism to another
	direct summand placed on its right. It follows that the morphism $p$, expressed
	with respect to the decomposition \eqref{eq:dec} is upper-triangular.  By
	\cite[Lemma~3.1]{vial-3-4},  the motive $M= (S,p,0)$ admits a weight
	decomposition $M=M^0\oplus M^1 \oplus M^2_{\mathrm{alg}} \oplus
	M^2_{\mathrm{tr}} \oplus M^3 \oplus M^4$, where each factor is isomorphic to a
	direct summand of the corresponding factor in the decomposition 
	\eqref{eq:dec}.
	In particular, this decomposition of $M$ inherits the properties of the
	decomposition \eqref{eq:dec}, \emph{e.g.}~$M^0$ and $M^4$ are of Tate type,
	$M^2_{\mathrm{alg}}$ becomes of Tate type after base-change to $k^s$ and
	$M^1$ and $M^3(1)$ are of curve type.

	\subsubsection{A weight argument}\label{subsec:weight}
	Thanks to Theorem \ref{thm:product}\eqref{ii}, $v(\mathcal{E})$ maps
	$\h^2_{\mathrm{tr}}(S)$ possibly non-trivially only in summands of 
	\begin{equation}\label{eq:sum}
	\bigoplus_{i=-2}^2
	\big( \h^0(S')(i) \oplus \h^1(S')(i) \oplus \h^2_{\mathrm{alg}}(S')(i) \oplus
	\h^2_{\mathrm{tr}}(S')(i) \oplus \h^3(S')(i) \oplus \h^4(S')(i)\big)
	\end{equation}
	of weight $\leq
	2$. Since $\Hom(\h^2_{\mathrm{tr}}(S),\mathds{1}(-1)) = \Hom(\mathds{1}(1),
	\h^2_{\mathrm{tr}}(S)^\vee) = \Hom(\mathds{1}(-1), \h^2_{\mathrm{tr}}(S)) = 
	0$, we see that
	$\h^2_{\mathrm{tr}}(S')$
	is the only direct summand of weight $2$ in \eqref{eq:sum} that admits a
	possibly non-trivial morphism from
	$\h^2_{\mathrm{tr}}(S)$. Likewise, the only direct summand of \eqref{eq:sum} of
	weight
	$\leq 2$ that maps possibly non-trivially in
	$\h^2_{\mathrm{tr}}(S)$ via $v(\mathcal{E}^\vee\otimes p^*\omega_S)$ is
	$\h^2_{\mathrm{tr}}(S')$. It
	follows that the restriction of $v_2(\mathcal{E}) $ induces an isomorphism
	$$\pi_{\mathrm{tr},S'}^2\circ v_2(\mathcal{E}) \circ \pi_{\mathrm{tr},S}^2:
	\h^2_{\mathrm{tr}}(S)
	\stackrel{\sim}{\longrightarrow}
	\h^2_{\mathrm{tr}}(S')$$  with inverse $\pi_{\mathrm{tr},S}^2\circ
	v_2(\mathcal{E}^\vee\otimes p^*\omega_S)\circ \pi_{\mathrm{tr}, S'}^2$\,; this
	fact
	will be used in the proof of Theorem~\ref{thm:main}.
	
	In a similar fashion, thanks to Theorem \ref{thm:product}\eqref{ii} and having
	in mind that $\h^1(S) \simeq \h^3(S)(1)$ is the direct summand of the motive of
	a curve (and similarly for $S'$), $v(\mathcal{E})$ induces isomorphisms
	$$\h^0(S) \oplus \h^2_{\mathrm{alg}}(S)(1) \oplus \h^4(S)(2)
	\stackrel{\sim}{\longrightarrow}
	\h^0(S') \oplus \h^2_{\mathrm{alg}}(S')(1) \oplus \h^4(S')(2)$$
	and 
	$$\h^1(S) \oplus  \h^3(S)(1) \stackrel{\sim}{\longrightarrow} \h^1(S')
	\oplus\h^3(S')(1).$$
	The first isomorphism yields an isomorphism  $\h^2_{\mathrm{alg}}(S) \simeq
	\h^2_{\mathrm{alg}}(S')$, while the second one yields, thanks to
	Theorem~\ref{thm:product}\eqref{i}, together with the semi-simplicity statement
	of Proposition~\ref{prop:curve}, isomorphisms $\h^1(S) \simeq
	\h^1(S')$ and $\h^3(S) \simeq \h^3(S')$. This finishes the proof of
	Theorem~\ref{thm:huybrechts}.
	\qed

	\subsection{A slight generalization to Theorem \ref{thm:huybrechts}}
	\label{subsec:3-4}
	The
	content of this paragraph won't be used in the proof of Theorem \ref{thm:main}.
	Recall that Theorem \ref{thm:huybrechts} fits more generally into the Orlov
	Conjecture~\ref{conj:orlov}.
	
	The method of proof of Theorem \ref{thm:huybrechts} can be pushed through to
	establish the following\,:
	
	\begin{prop}\label{prop:generalization}
		Let $X$ and $Y$ be two smooth projective varieties of dimension 3 or 4 over a
		field $k$. Assume either of the following\,:
		\begin{itemize}
			\item $\dim X = 3$ and $\CH_0(X)$ is representable\,;
			\item $\dim X=4$, $\CH_0(X)$ and $\CH_0(Y)$ are both representable, and $X$
			and
			$Y$ have same Picard rank.
		\end{itemize}
		Then $\mathrm{D}^b(X)\simeq \mathrm{D}^b(Y)$ implies that $ \mathfrak{h}(X)
		\simeq
		\mathfrak{h}(Y)$.
	\end{prop}
	
	Here, we say that a smooth projective $k$-variety $X$ of dimension $d$ has
	\emph{representable} $\CH_0$ if for a choice of universal domain (\emph{i.e.},
	algebraically closed field of infinite transcendence degree over its prime
	subfield) $\Omega$ containing $k$, there exists a smooth projective
	$\Omega$-curve $C$ and a correspondence $\gamma \in \Hom(\h(X_\Omega),\h(C))$
	such that $\gamma^*\CH_0(C) = \CH_0(X_\Omega)$. Examples of such varieties
	include varieties with maximally rationally connected quotient of dimension
	$\leq 1$, and in particular rationally connected varieties.
	
	\begin{proof}
		We start with the case of threefolds.
		By \cite{gg}, $X$ admits a Chow--K\"unneth
		decomposition, where the even-degree summands are of Tate
		type, while the odd-degree summands are Tate twists of motives of
		curve type. The arguments of \S \ref{subsec:orlov} show that
		$\h(Y)$ is a direct
		summand of $\bigoplus_{i=-3}^3 \h(X)(i)$\,; in particular, by
		Kimura finite-dimensionality
		(or by Theorem \ref{thm:product}\eqref{ii} together with \cite[Lemma
		3.1]{vial-3-4} as used in \S \ref{subsec:kmp}), $\h(Y)$ has a Chow--K\"unneth
		decomposition
		with
		a similar property to that of $X$ (and hence has representable $\CH_0$). The
		arguments of \S \ref{subsec:weight} provide isomorphisms
		$$\h^0(X) \oplus \h^2(X)(1) \oplus \h^4(X)(2) \oplus \h^6(X)(3) \simeq 
		\h^0(Y)
		\oplus \h^2(Y)(1) \oplus \h^4(Y)(2) \oplus \h^6(Y)(3)$$ and
		\begin{equation}\label{eq:odd}
		\h^1(X) \oplus \h^3(X)(1) \oplus \h^5(X)(2)  \simeq  \h^1(Y) \oplus
		\h^3(Y)(1)
		\oplus \h^5(Y)(2),
		\end{equation}
		Since the even-degree summands are of Tate type, and since $\dim
		\HH^{2i}(X) = \dim \HH^{6-2i}(X)$ by Poincar\'e duality (and similarly for
		$Y$),
		we conclude that
		$\h^{2i}(X) \simeq \h^{2i}(Y)$ for all $i$. Now, a theorem of Popa--Schnell
		\cite{ps} says that two derived equivalent complex varieties have isogenous
		reduced Picard scheme (see \cite[Appendix]{honigs} for the case of
		$k$-varieties). It follows from Theorem~\ref{thm:product}\eqref{i}  that
		$\h^1(X) \simeq \h^1(Y)$, and then by duality that $\h^5(X) \simeq \h^5(Y)$.
		Since all terms of \eqref{eq:odd} are of curve type, we deduce
		from the semi-simplicity statement of Proposition~\ref{prop:curve} that
		$\h^3(X)\simeq \h^3(Y)$. Alternately, from \cite{acmv}, two
		derived equivalent threefolds have degree-wise isomorphic cohomology groups
		(the isomorphisms being induced by some algebraic correspondences)\,;
		it then follows from the description of the motives of $X$ and $Y$ together
		with
		Proposition~\ref{prop:curve} that $X$ and $Y$ have isomorphic motives.
		
		In the case of fourfolds, we first note by \cite[Theorem
		3.11]{vial-abelian}  that the motive of a fourfold
		$X$ with representable $\CH_0$ admits a decomposition of the form
		$$\h(X) \simeq (C,p,0) \oplus (S,q,1) \oplus (C,{}^tp,3),$$ for some curve $C$
		and some surface $S$.
		It follows from the arguments of \S \ref{subsec:kmp} that $\h(X)$ admits a
		Chow--K\"unneth decomposition such that
		$\h^{2i+1}(X)(i)$ is of curve type for all $i$,
		$\h^{2i}(X)$ is of Tate type for all $i\neq 2$, and
		$\h^4(X)$ further decomposes 
		as $\h^4_{\mathrm{alg}}(X)
		\oplus
		\h^4_{\mathrm{tr}}(X)$, with the property that $\h^4_{\mathrm{alg}}(X)$
		becomes, after base-change to $k^s$, a direct sum of
		Tate motives $\mathds{1}(-2)$ and $\h^4_{\mathrm{tr}}(X)(1)$ is a direct
		summand
		of the
		motive of a surface with $\Hom(\h^4_{\mathrm{tr}}(X),\mathds{1}(-i))=0$ for
		$i\neq 3$.
		The arguments of
		\S \ref{subsec:weight} then provides isomorphisms
		\begin{small}
			$$\h^0(X) \oplus \h^2(X)(1) \oplus \h_{\mathrm{alg}}^4(X)(2) \oplus
			\h^6(X)(3)\oplus
			\h^8(X)(4) \simeq  \h^0(Y) \oplus \h^2(Y)(1) \oplus \h^4_{\mathrm{alg}}(Y)(2)
			\oplus
			\h^6(Y)(3)\oplus \h^8(Y)(4),$$
			$$\h^1(X) \oplus \h^3(X)(1) \oplus \h^5(X)(2) \oplus \h^7(X)(3)  \simeq 
			\h^1(Y)
			\oplus \h^3(Y)(1) \oplus \h^5(Y)(2) \oplus \h^7(Y)(3)$$ 
			$$\text{and} \quad \h^4_{\mathrm{tr}}(X) \simeq \h^4_{\mathrm{tr}}(Y).
			\qquad$$
		\end{small}
		As in the case of threefolds, since the Picard numbers of $X$ and $Y$ agree,
		we conclude
		that
		$\h^{2i}(X) \simeq \h^{2i}(Y)$ for all $i\neq 2$ and that
		$\h^4_{\mathrm{alg}}(X)
		\simeq
		\h^4_{\mathrm{alg}}(Y)$, while by utilizing the Theorem of Popa--Schnell
		\cite{ps}, we
		conclude from Theorem~\ref{thm:product}\eqref{i} 
		that $\h^1(X) \simeq \h^1(Y)$ and then by duality that $\h^7(X)\simeq
		\h^7(Y)$.
		It follows by cancellation (Proposition~\ref{prop:curve}) that $\h^3(X) \oplus
		\h^5(X)(1) \simeq \h^3(Y)
		\oplus
		\h^5(Y)(1)$. Since there is a Lefschetz isomorphism $\HH^3(X)\simeq 
		\HH^5(X)(1)$ and
		similarly for~$Y$, we
		conclude (again from Proposition~\ref{prop:curve}) that $\h^3(X)\simeq
		\h^3(Y)$ and  $\h^5(X)\simeq \h^5(Y)$.
	\end{proof}

	\section{Motives of varieties as Frobenius algebra objects}\label{sec:Frob}
	\subsection{Algebras and Frobenius algebras}\label{SS:alg}
	Let $\mathcal{C}$ be a symmetric monoidal category with tensor unit
	$\mathds{1}$.
	An \emph{algebra object} in  $\mathcal{C}$ is an object $M$ together with a
	unit morphism $\eta: \mathds{1}\to M$ and a multiplication morphism $\mu:
	M\otimes M\to M$ satisfying the unit axiom $\mu\circ (\id\otimes
	\eta)=\id=\mu\circ (\eta\otimes \id)$ and the associativity axiom
	$\mu\circ(\mu\otimes \id)=\mu\circ (\id\otimes \mu)$. It is called
	\emph{commutative} if moreover $\mu=\mu\circ c_{M, M}$ is satisfied, where
	$c_{M, M}$ is the commutativity constraint of the category~$\mathcal{C}$. A
	\emph{morphism of algebra objects} between two algebra objects $M$ and $N$ is a
	morphism $\phi: M\to N$ in $\mathcal{C}$ that preserves the multiplication $\mu$
	and the unit $\eta$. We note that an algebra structure on an object $M$ of
	$\mathcal{C}$ induces naturally an algebra structure on the $n$-th tensor powers
	$M^{\otimes n}$ of $M$, and that a morphism $\phi: M \to N$ of algebra objects
	induces naturally a morphism of algebra objects $\phi^{\otimes n} : M^{\otimes
		n} \to N^{\otimes n}$ which is an isomorphism if $\phi$ is.\medskip

	If $\mathcal{C}$ is moreover rigid\footnote{A symmetric monoidal category $\mathcal{C}$ is called \textit{rigid} if each object $M$ admits a dual $M^\vee$, together with morphisms $\mathds{1}\to M\otimes M^\vee$ and $M^\vee\otimes M\to \mathds{1}$ satisfying natural axioms\,; see \emph{e.g.}~\cite[Section 2.2]{andre}.} and possesses a $\otimes$-invertible object (that is an object $L$ such that $L\otimes L^{\vee}\simeq \mathds{1}$),
	then we can speak of Frobenius algebra objects in~$\mathcal{C}$\,:
	\begin{defn}[Frobenius algebra objects]\label{def:FrobAlg}
		Let $(\mathcal{C}, \otimes, \vee, \mathds{1})$ be a rigid symmetric monoidal
		category admitting a $\otimes$-invertible object denoted by $\mathds{1}(1)$.
		Let
		$d$ be an integer. A \emph{degree}-$d$ \emph{Frobenius algebra object} in
		$\mathcal{C}$ is the data of an object $M\in \mathcal{C}$ endowed with 
		\begin{itemize}
			\item $\eta: \mathds{1}\to M$, a unit morphism\,;
			\item $\mu: M\otimes M\to M$, a multiplication morphism\,;
			\item $\lambda: M^{\vee}\stackrel{\sim}{\longrightarrow} M(d)$, an
			isomorphism,
			called the \emph{Frobenius structure}\,;
		\end{itemize}
		satisfying the following axioms\,:
		\begin{enumerate}[$(i)$]
			\item  (Unit) $\mu\circ (\id\otimes \eta)=\id=\mu\circ (\eta\otimes \id)$\,;
			\item (Associativity) $\mu\circ(\mu\otimes \id)=\mu\circ (\id\otimes \mu)$\,;
			\item (Frobenius condition) $(\id\otimes \mu)\circ (\delta\otimes
			\id)=\delta\circ \mu=(\mu\otimes \id)\circ (\id\otimes \delta)$,
		\end{enumerate}
		where the \emph{comultiplication} morphism $\delta: M\to M\otimes M(d)$ is
		defined by dualizing $\mu$ via the following commutative diagram\,:
		\begin{equation*}
		\xymatrix{
			M^{\vee}\ar[r]^-{{}^{t}\mu} \ar[d]_{\lambda}^{\simeq} & M^{\vee}\otimes
			M^{\vee}\ar[d]_{\lambda\otimes \lambda}^{\simeq}\\
			M(d) \ar[r]^-{\delta(d)}& M(d)\otimes M(d)
		}
		\end{equation*}
		We define also the \emph{counit morphism} $\epsilon: M\to \mathds{1}(-d)$ by
		dualizing $\eta$ via the following diagram:
		\begin{equation*}
		\xymatrix{
			M^{\vee}\ar[r]^{{}^{t}\eta}\ar[d]^{\lambda}_{\simeq}& \mathds{1}\\
			M(d) \ar[ur]_{\epsilon(d)}&
		}
		\end{equation*}
		We remark that $\epsilon$ and $\delta$ satisfy automatically the counit and
		coassociativity axioms. 
		
		A Frobenius algebra object $M$ is called \emph{commutative} if the underlying
		algebra object is commutative\,: $\mu\circ c_{M,M}=\mu$. Commutativity is
		equivalent to the cocommutativity of $\delta$.
		The morphism $\beta=\epsilon\circ \mu: M\otimes M\to \mathds{1}(-d)$, called
		the
		\emph{Frobenius pairing}, is also sometimes used. It is a symmetric pairing if
		$M$ is commutative. 
	\end{defn}

	\begin{rmk}
		In the case of Frobenius algebra objects of degree 0, the $\otimes$-invertible
		object $\mathds{1}(1)$ is not needed in the definition, and it is reduced to
		the
		usual notion of Frobenius algebra object in the literature (see for example
		\cite{Abrams}, \cite{Kock}). In this sense, Definition \ref{def:FrobAlg}
		generalizes  the existing definition of Frobenius structure by allowing
		non-zero
		twists. We believe that our more flexible notion is necessary and adequate for
		more sophisticated tensor categories than that of vector spaces, such as the
		categories of Hodge structures, Galois representations, motives, \emph{etc}.
	\end{rmk}
	
	\begin{rmk}[Morphisms]\label{rmk:MorIsom}
		Morphisms of Frobenius algebra objects are defined in the natural way, that
		is,
		as morphisms $\phi: M\to N$ such that all the natural diagrams involving the
		structural morphisms are commutative. In particular, in order to admit
		non-trivial morphisms, the degrees of the Frobenius algebra objects $M$ and $N$
		must coincide and the following
		diagram is then commutative\,:
		\begin{equation*}
		\xymatrix{
			N^{\vee} \ar[r]^{{}^{t}\phi} \ar[d]^{\simeq}_{\lambda_{N}}&
			M^{\vee}\ar[d]_{\simeq}^{\lambda_{M}}\\
			N(d) & M(d)\ar[l]^{\phi(d)}
		}
		\end{equation*}
		As a result, all morphisms between Frobenius algebra objects are in fact 
		invertible. It is an exercise to show that an isomorphism $\phi: M\to N$
		between two Frobenius algebra objects respects the Frobenius algebra structures
		if and only if it is compatible with the algebra structure (\emph{i.e.} with
		$\mu$) and the Frobenius structure (\emph{i.e.} with $\lambda$). This is proved
		in Proposition~\ref{prop:RigidAlgIso} in the case of Chow motives of smooth
		projective varieties. In addition, $\phi^{\otimes n} : M^{\otimes n} \to
		N^{\otimes n}$ is naturally an isomorphism of Frobenius algebra objects, as is
		the  dual ${}^t\phi : N^\vee \to M^\vee$.
	\end{rmk}
	
	We summarize the above discussion in the following 
	\begin{lem}\label{lemma:IsoFrobAlg}
		Let $M, N$ be two Frobenius algebra objects of degree $d$. A morphism of
		algebra
		objects $\phi: M\to N$ is a morphism of Frobenius algebra objects if and only
		if
		it is an isomorphism and it is \emph{orthogonal} in the sense that
		$\phi(d)^{-1}=\lambda_{M}\circ {}^{t}\phi\circ \lambda_{N}^{-1}$, or more
		succinctly, $\phi^{-1}={}^{t}\phi$.
	\end{lem}
	\begin{proof}
		The ``if'' part follows from the definition. The ``only if'' part is explained
		in Remark~\ref{rmk:MorIsom}.
	\end{proof}

	Now let us turn to important examples of Frobenius algebra objects.
	
	\begin{ex}[Cohomology as a graded vector space]\label{ex:Cohomology}
		Let $X$ be a connected compact orientable (real) manifold of dimension $d$.
		Then
		its cohomology group $\HH^{*}(X, \Q)$ is naturally a Frobenius algebra object
		of
		degree $d$ in the category of $\Z$-graded $\Q$-vector spaces (where morphisms
		are \emph{degree-preserving} linear maps and the $\otimes$-invertible object
		is
		chosen to be $\Q[1]$, the 1-dimensional vector space sitting in degree $-1$).
		The unit morphism $\eta: \Q\to \HH^{*}(X, \Q)$ is given by the fundamental
		class
		$[X]$\,; the multiplication morphism $\mu: \HH^{*}(X, \Q)\otimes \HH^{*}(X,
		\Q)\to \HH^{*}(X, \Q)$ is the cup-product\,; the Frobenius structure comes
		from
		the Poincar\'e duality $$\lambda: \HH^{*}(X,
		\Q)^{\vee}\stackrel{\sim}{\longrightarrow} \HH^{*}(X, \Q)[d]=\HH^{*}(X,
		\Q)\otimes \Q[d].$$
		The induced comultiplication morphism $\delta: \HH^{*}(X, \Q)\to \HH^{*}(X,
		\Q)\otimes \HH^{*}(X, \Q)[d]$ is the Gysin map for the diagonal embedding
		$X\hookrightarrow X\times X$\,; the counit morphism $\epsilon: \HH^{*}(X,
		\Q)\to
		\Q[-d]$ is the integration $\int_{X}$. The Frobenius condition is a classical
		exercise. Note that $\HH^{*}(X, \Q)$ is commutative, because the commutativity
		constraint in the category of graded vector spaces is the super one.
		
		If instead we consider the cohomology group as merely an ungraded vector
		space,
		then it becomes a Frobenius algebra object of degree 0 (\emph{i.e.} in the
		usual
		sense)\,; this is one of the main examples in the literature.
	\end{ex}
	
	\begin{ex}[Hodge structures]\label{ex:Hodge}
		A pure (rational) Hodge structure is a finite-dimensional $\Z$-graded
		$\Q$-vector space $H=\oplus_{n\in Z}H^{(n)}$ such that each $H^{(n)}$ is given
		a
		Hodge structure of weight $n$. A morphism between two Hodge structures is
		required to preserve the weights. The category of pure Hodge structures is
		naturally a rigid symmetric monoidal category. The $\otimes$-invertible object
		is chosen to be $\Q(1)$, which is the 1-dimensional vector space $(2\pi
		i)\cdot
		\Q$ with Hodge structure purely of type $(-1, -1)$.
		
		Let $X$ be a compact K\"ahler manifold of (complex) dimension $d$. Then
		$\HH^{*}(X, \Q)$ is naturally a commutative Frobenius algebra object of degree
		$d$ in the category of pure $\Q$-Hodge structures. The structural morphisms
		are
		the same as in Example \ref{ex:Cohomology} up to replacing $[d]$ by $(d)$. For
		instance, 
		$\lambda: \HH^{*}(X, \Q)^{\vee}\stackrel{\sim}{\longrightarrow} \HH^{*}(X,
		\Q)(d)$.
	\end{ex}

	\subsection{Frobenius algebra structure on the motives of
		varieties} 
	The category of rational Chow motives over a field $k$ is rigid and symmetric
	monoidal. We choose the $\otimes$-invertible object to be the Tate motive
	$\mathds{1}(1)$. Then for any smooth projective $k$-variety $X$ of dimension
	$d$, its Chow motive $\h(X)$ is naturally a commutative Frobenius algebra
	object
	of degree $d$ in the category of Chow motives. Let us explain the structural
	morphisms in detail.

	Let $\delta_X$
	denote the class of the small diagonal $\{(x,x,x) : x\in X\}$ in $\CH_d(X\times
	X\times X)$. Note that for $\alpha$ and $\beta$ in $\CH^*(X)$, we have
	$(\delta_X)_*(\alpha\times \beta) = \alpha\cdot \beta$, so that $\delta_X$ seen
	as an element of $\Hom(\h(X)\otimes \h(X) , \h(X))$ describes the intersection
	theory on the Chow ring of $X$, as well as the cup product of its cohomology
	ring. So it is natural to define the multiplication morphism $$\mu : \h(X)
	\otimes \h(X) \longrightarrow \h(X)$$ to be the one given by the small diagonal
	$\delta_{X}\in \CH^{2d}(X\times X\times X)=\Hom(\h(X)\otimes\h(X), \h(X))$\,;
	it
	can be checked to be commutative and associative. 
	The unit morphism $\eta: \mathds{1}\to \h(X)$ is again given by the fundamental
	class of $X$. The unit axiom is very easy to check.
	
	The Frobenius structure is defined as the following canonical isomorphism,
	called the motivic Poincar\'e duality, given by the class of diagonal
	$\Delta_{X}\in \CH^{d}(X\times X)=\Hom(\h(X)^{\vee}, \h(X)(d))$\,:
	$$\lambda: \h(X)^{\vee}\stackrel{\sim}{\longrightarrow} \h(X)(d).$$
	One readily checks that the induced comultiplication morphism 
	$$\delta: \h(X)\to \h(X)\otimes \h(X)(d)$$
	is given by the small diagonal $\delta_{X}\in \CH^{2d}(X\times X\times
	X)=\Hom(\h(X), \h(X)\otimes \h(X)(d))$, while the counit morphism $$\epsilon:
	\h(X)\to \mathds{1}(-d)$$ is given by the fundamental class.
	The following lemma proves that, endowed with these structural morphisms,
	$\h(X)$ is indeed a Frobenius algebra object.
	\begin{lem}[Frobenius condition]
		Notation is as above. We have an equality of endomorphisms of $\h(X)\otimes
		\h(X)$\,:
		$$(\id\otimes \mu)\circ (\delta\otimes \id)=\delta\circ \mu=(\mu\otimes
		\id)\circ (\id\otimes \delta).$$
	\end{lem}
	\begin{proof}
		We only show $\delta\circ \mu=(\mu\otimes \id)\circ (\id\otimes \delta)$, the
		other equality being similar.
		We have a commutative cartesian diagram without excess intersection\,:
		\begin{equation*}
		\xymatrix{
			X \ar[r]^{\Delta}\ar[d]_{\Delta}& X\times X \ar[d]^{\Delta\times \id}\\
			X\times X \ar[r]_-{\id\times \Delta}& X\times X\times X,
		}
		\end{equation*}
		where $\Delta : X \to X\times X$ denotes the diagonal embedding.
		The base-change formula yields $$(\Delta\times \id)^{*}\circ (\id\times
		\Delta)_{*}=\Delta_{*}\circ\Delta^{*}$$
		on Chow groups, hence also for Chow motives by Manin's identity principle
		\cite[\S 4.3.1]{andre}.  Now it suffices to notice that $\Delta_{*}$ is the
		comultiplication $\delta$ and $\Delta^{*}$ is the multiplication $\mu$.
	\end{proof}
	
	\begin{rmk}
		In general, a tensor functor $F: \mathcal{C}\to \mathcal{C'}$ between two
		rigid
		symmetric monoidal categories sends a Frobenius algebra object in
		$\mathcal{C}$
		to  such an object in $\mathcal{C}'$.
		Example \ref{ex:Hodge} is obtained by applying the Betti realization functor
		from the category of Chow motives to that of pure Hodge structures\,; Example
		\ref{ex:Cohomology} (for K\"ahler manifolds) is obtained by further applying
		the
		forgetful functor ($\Q(1)$ is sent to $\Q[2]$).
	\end{rmk}

	\subsection{(Iso)morphisms of Chow motives as Frobenius algebra objects}
	The notion of morphisms between two algebra objects is the natural one. Let us
	spell it out for motives of varieties.
	A non-zero morphism $\Gamma : \h(X) \rightarrow \h(Y)$ between the motives of
	two smooth
	projective varieties over a field~$k$ is said to \emph{preserve the
		algebra structures} if
	the following diagram commutes\footnote{As explained in Lemma~\ref{lem:formal},
		a non-zero morphism between algebra objects that preserves the multiplication
		morphisms
		must also preserve the unit morphisms, and hence is a morphism of algebra
		objects in the sense of~\S \ref{SS:alg}.}
	\begin{equation}\label{eq:diagmult}
	\xymatrix{\h(X)\otimes \h(X) \ar[rr]^{\quad \delta_X} \ar[d]^{\Gamma\otimes
			\Gamma} && \h(X) \ar[d]^\Gamma \\ 
		\h(Y)\otimes \h(Y) \ar[rr]^{\quad \delta_Y} && \h(Y).
	}
	\end{equation}
	For example, if $f: Y \to X$ is a $k$-morphism, then $f^* : \h(X) \to \h(Y)$
	is
	a morphism of algebra objects. 
	Note that if $\Gamma : \h(X) \rightarrow \h(Y)$ preserves the
	algebra structures, then  $\Gamma_* : \CH^*(X) \to \CH^*(Y)$ is a $\Q$-algebra
	homomorphism. In fact, since in that case  $\Gamma^{\otimes n} : \h(X^n)
	\rightarrow \h(Y^n)$
	also preserves the algebra structures for all $n>0$, $(\Gamma^{\otimes
		n})_* : \CH^*(X^n) \rightarrow \CH^*(Y^n)$ is also a $\Q$-algebra
	homomorphism.
	We say that the Chow motives of $X$ and $Y$ are \emph{isomorphic as
		algebra objects} if there exists an isomorphism $\Gamma : \h(X) \rightarrow
	\h(Y)$ that preserve the algebra structures. 
	The following lemma is a formal consequence of
	the definition.

	\begin{lem}[Algebra morphisms]\label{lem:formal}
		Let $X$ and $Y$ be connected smooth projective varieties and let $\Gamma :
		\h(X)
		\rightarrow \h(Y)$ be a non-zero morphism that preserves the algebra
		structures.
		\begin{enumerate}[(i)]
			\item $\Gamma$ preserves the units\,:  if $[X]$ is  the fundamental class of
			$X$
			in
			$\CH^0(X)$ and similarly for~$Y$, then $$\Gamma_*[X] = [Y].$$
			\item Suppose  $X$ and $Y$ have same dimension and define $c$ to be the
			rational number such that $\Gamma^*[Y]=c\, [X]$\,; then  $$(\Gamma \otimes
			\Gamma)^* \Delta_Y= c\, \Delta_{X}.$$
			In particular, $\Gamma$ is an
			isomorphism if and only if $c\neq 0$, and in this case, due to Lieberman's
			formula\footnote{see \emph{e.g.}~\cite[\S 3.1.4]{andre}, and
				\cite[Lemma~3.3]{vial-abelian} for a proof.}, the inverse of $\Gamma$ is
			equal
			to $\frac{1}{c}^t \Gamma$. 
		\end{enumerate}
	\end{lem}
	\begin{proof}
		$(i)$ This is the analogue of the basic fact that a non-trivial homomorphism
		of unital algebras preserves the units. Concretely, the fundamental class of
		$X$
		provides a morphism $1_X : \mathds{1} \to
		\h(X)$, and similarly for~$Y$, and we need to show that $\Gamma\circ 1_X =
		1_Y$.
		First, for dimension reasons we have $\Gamma\circ 1_X = \lambda\cdot 1_Y$ for
		some $\lambda \in \Q$. Compose then the diagram \eqref{eq:diagmult} with the
		morphism $1_X\otimes 1_X$\,; one obtains $\lambda^2 = \lambda$. If $\lambda =
		0$, then by composing diagram \eqref{eq:diagmult} with the morphism
		$1_X\otimes
		\mathrm{id}_X$, we find that $\Gamma =0$.
		Hence $\lambda =1$ and we are done.
		
		$(ii)$ 	The commutativity of \eqref{eq:diagmult} provides the identity
		$\Gamma\circ \delta_X = \delta_Y\circ (\Gamma\otimes \Gamma)$. Letting the
		latter act contravariantly on $[Y]$ yields
		$$c\, \Delta_X = (\Gamma\otimes \Gamma)^*\Delta_Y = {}^t\Gamma\circ \Gamma,$$ 
		where $c$ is the rational number such that $\Gamma^*[Y] = c\, [X]$ and where
		the second equality is Lieberman's formula. Since we assume that $\Gamma$ is
		invertible, we get that $\Gamma^{-1}=\frac{1}{c}\, {}^t\Gamma$.
	\end{proof}

	As is alluded to in Lemma \ref{lemma:IsoFrobAlg}, the notion of orthogonality
	is highly relevant when considering morphisms between Frobenius algebras. Let
	us
	recast it in the context of motives\,:

	\begin{defn}[Orthogonal isomorphisms]\label{def:OrthIso}
		Let $X$ and $Y$ be two smooth projective varieties of the same dimension and
		$\Gamma: \h(X)\to \h(Y)$ be an isomorphism between their Chow motives. Then by
		Lieberman's formula we see that $\Gamma^{-1}={}^{t}\Gamma$ if and only if
		$(\Gamma\otimes \Gamma)_{*}\Delta_{X}=\Delta_{Y}$. In this case, $\Gamma$ is
		called an \emph{orthogonal} isomorphism. 
	\end{defn}

	Finally, we can unravel the meaning of being isomorphic as Frobenius algebra
	objects for the motives of two varieties (and the same holds for Hodge
	morphisms between the cohomology algebras of smooth projective varieties of same
	dimension).
	
	\begin{prop}[Frobenius algebra isomorphisms]\label{prop:RigidAlgIso}
		Let $X$ and $Y$ be two smooth projective varieties of the same dimension and
		$\Gamma: \h(X)\to \h(Y)$ be a morphism between their motives. Then the
		following
		are equivalent:
		\begin{enumerate}[$(i)$]
			\item $\Gamma$ is an isomorphism of Frobenius algebra objects.
			\item $\Gamma$ is an algebra isomorphism and $\Gamma$ is
			\emph{orthogonal}\,: that is, $\Gamma^{-1}={}^{t}\Gamma$ or equivalently,
			$(\Gamma\otimes \Gamma)_{*}\Delta_{X}=\Delta_{Y}$.
			\item $\Gamma$ is an algebra isomorphism and $\deg(\Gamma)=1$\,: that is,
			$\deg (\Gamma_*[pt])=1$ or equivalently,
			$\Gamma^{*}[Y]=[X]$.
			\item $\Gamma$ is an isomorphism and $(\Gamma\otimes
			\Gamma)_{*}\Delta_{X}=\Delta_{Y}$ and $(\Gamma\otimes \Gamma\otimes
			\Gamma)_{*}\delta_{X}=\delta_{Y}$.
		\end{enumerate}	
	\end{prop}
	\begin{proof}
		The equivalence between $(i)$ and $(ii)$ is a special case of Lemma
		\ref{lemma:IsoFrobAlg}. 
		The equivalence between $(ii)$ and $(iii)$ can be read off
		Lemma~\ref{lem:formal}$(ii)$.
		For the equivalence between $(ii)$ and $(iv)$, one only
		needs to see that  an orthogonal isomorphism ($\Gamma^{-1}={}^{t}\Gamma$) is
		an
		algebra morphism ($\Gamma \circ \delta_X =	\delta_{Y} \circ (\Gamma \otimes
		\Gamma)$) if and only if $(\Gamma\otimes \Gamma\otimes
		\Gamma)_{*}\delta_{X}=\delta_{Y}$. But this again follows from  Lieberman's
		formula.
	\end{proof}

	\section{Derived equivalent K3 surfaces}
	
	The aim of this section is to prove Theorem \ref{thm:main}, as well as Corollaries
	\ref{cor:Powers} and
	\ref{cor:torelli}.

	\subsection{Proof of Theorem \ref{thm:main}}\label{S:mainthm}
	The proof relies
	crucially on the Beauville--Voisin description of the algebra structure on the
	motive of K3 surfaces\,:
	
	\begin{thm}[Beauville--Voisin \cite{bv}] \label{thm:bv}
		Let $S$ be a K3 surface and let $o_S$ be the class of any point lying on a
		rational curve on $S$. Then,
		as cycle classes in $\CH_2(S\times S \times S)$, we have
		\begin{equation}\label{eq:bv}
		\delta_S = p_{12}^*\Delta_S\cdot p_3^*o_S + p_{13}^*\Delta_S \cdot p_2^*o_S
		+
		p_{23}^*\Delta_S\cdot p_1^*o_S - p_1^*o_S\cdot p_2^*o_S - p_1^*o_S\cdot
		p_3^*o_S
		- p_2^*o_S\cdot p_3^*o_S,
		\end{equation}
		where $p_k : S\times S \times S \to S$ and $p_{ij}: S\times S \times S \to S
		\times S$ denote the various projections.
	\end{thm}
	Note that, for a K3 surface $S$, Theorem~\ref{thm:bv} implies
	that $\alpha\cdot \beta = \deg (\alpha\cdot\beta)\,o_S$ for all divisors
	$\alpha, \beta
	\in
	\CH^1(S)$, and that $c_2(S) = (\delta_S)_*\Delta_S = \chi(S)\,o_S = 24\, o_S
	\in
	\CH^2(S)$. 
	(Of course, this is due originally to Beauville--Voisin \cite{bv}.)\medskip

	According to Proposition~\ref{prop:RigidAlgIso}, in order to establish
	Theorem~\ref{thm:main}, it is necessary and sufficient to produce a
	correspondence $\Gamma : \h(S) \to
	\h(S')$ which is invertible and such that 
	\begin{enumerate}
		\item[$(i)$] $(\Gamma\otimes \Gamma)_*\Delta_S = \Delta_{S'}$, or equivalently
		$\Gamma^{-1} = {}^t\Gamma$\,;
		\item[$(ii)$] $(\Gamma\otimes \Gamma\otimes \Gamma)_*\delta_S = \delta_{S'}$.
	\end{enumerate}
	By the Beauville--Voisin Theorem~\ref{thm:bv}, it is sufficient to
	produce a correspondence $\Gamma : \h(S) \to
	\h(S')$ which is invertible and such that 
	\begin{enumerate}
		\item[$(i)$] $(\Gamma\otimes \Gamma)_*\Delta_S = \Delta_{S'}$, or equivalently
		$\Gamma^{-1} = {}^t\Gamma$\,;
		\item[$(ii')$] $\Gamma_*o_S = o_{S'}$.
	\end{enumerate}
	For the sake of completeness, we note that the above is also 
	necessary\,: Indeed, by Lieberman's formula \cite[\S 3.1.4]{andre}, 
	$(ii)$ is equivalent to $\delta_{S'}=\Gamma\circ \delta_S\circ ({}^t\Gamma\otimes {}^t\Gamma)$, from which we obtain
	\[24o_{S'}=\delta_{S', *}(\Delta_{S'})=\Gamma_*\circ \delta_{S, *}\circ ({}^t\Gamma\otimes {}^t\Gamma)_*(\Delta_{S'})=\Gamma_*\circ \delta_{S, *}(\Delta_S)=\Gamma_*(24o_S),\]
	where the first and last equalities are due to Beauville and Voisin as mentioned above, and where the third equality uses $(i)$.

	We now proceed to the proof of Theorem~\ref{thm:main}, \emph{i.e.} to
	constructing an invertible correspondence satisfying $(i)$ and $(ii')$ above.
	Given a K3 surface $S$, we consider the refined Chow--K\"unneth decomposition
	of
	Kahn--Murre--Pedrini as described in \S \ref{subsec:kmp} given by $$\h(S) =
	\h^0(S) \oplus \h^2_{\mathrm{alg}}(S) \oplus
	\h^2_{\mathrm{tr}}(S) \oplus \h^4(S),$$
	with $\pi^0_S = o_S\times S$ and $\pi^4_S = S\times o_S$, where $\pi_S^i$
	denote the projectors on the corresponding direct summands and where $o_S$
	denotes the Beauville--Voisin zero-cycle as in Theorem \ref{thm:bv}. Moreover,
	the decomposition is such that  ${}^t	\pi^2_{\mathrm{alg},S} = 
	\pi^2_{\mathrm{alg},S} $ and ${}^t	\pi^2_{\mathrm{tr},S} = 
	\pi^2_{\mathrm{tr},S}$.

	Consider now two twisted derived equivalent K3 surfaces $S$ and $S'$. As in the
	proof of Theorem~\ref{thm:huybrechts}, we only give a proof in the untwisted
	case, the twisted case being similar. We fix an exact linear equivalence
	$\Phi_{\mathcal{E}} : \mathrm{D}^b(S) \stackrel{\sim}{\longrightarrow}
	\mathrm{D}^b(S')$ with Fourier--Mukai kernel $\mathcal{E} \in
	\mathrm{D}^b(S\times S')$. The proof will proceed in two steps. First, we will
	construct an invertible correspondence $$\Gamma_{\mathrm{alg}} :
	\h_{\mathrm{alg}}(S) \to \h_{\mathrm{alg}}(S') \quad \text{with} \
	(\Gamma_{\mathrm{alg}})^{-1} = {}^t \Gamma_{\mathrm{alg}} \ \text{and} \
	(\Gamma_{\mathrm{alg}})_* o_S = o_{S'},$$ where $\h_{\mathrm{alg}}(S) = \h^0(S)
	\oplus \h_{\mathrm{alg}}^2(S) \oplus \h^4(S)$ (and similarly for $S'$) is the
	algebraic summand of the motive of $S$\,; second, we will construct an
	invertible correspondence on the transcendental summands of the motives of $S$
	and $S'$\,: $$\Gamma_{\mathrm{tr}} : \h^2_{\mathrm{tr}}(S) \to
	\h^2_{\mathrm{tr}}(S') \quad \text{with} \ (\Gamma_{\mathrm{tr}})^{-1} = {}^t
	\Gamma_{\mathrm{tr}} \ \text{and} \ (\Gamma_{\mathrm{tr}})_* o_S = 0.$$
	The correspondence $$\Gamma := \Gamma_{\mathrm{alg}} + \Gamma_{\mathrm{tr}} :
	\h(S) \to \h(S')$$ will then provide the desired isomorphism of Frobenius
	algebra objects.\medskip

	First, the numerical Grothendieck group $K_{0}^{\operatorname{num}}$ equipped
	with the Euler pairing is clearly a derived invariant. Using the Chern
	character
	isomorphism, we obtain an isometry between the quadratic spaces
	$\widetilde{\mathrm{NS}}(S_{k^s})_\Q$ and
	$\widetilde{\mathrm{NS}}(S'_{k^s})_\Q$, where $\widetilde{\mathrm{NS}}$ is the
	extended N\'eron--Severi group equipped with the Mukai pairing, hence is
	isometric to the (orthogonal) direct sum of the N\'eron--Severi lattice
	(endowed
	with the intersection pairing) and a copy of the hyperbolic plane. By Witt's
	cancellation theorem, 
	the N\'eron--Severi groups $\mathrm{NS}(S_{k^s})_\Q$ and
	$\mathrm{NS}(S'_{k^s})_\Q$ of two derived equivalent surfaces are isomorphic
	both as $\mathrm{Gal}(k)$-representations and as quadratic spaces\,;
	there exists therefore a
	correspondence $M = \pi^2_{\mathrm{alg},S'} \circ M \circ
	\pi^2_{\mathrm{alg},S}$ in $\CH^2(S\times_k S')$ inducing an
	isometry $\mathrm{NS}(S_{k^s})_\Q \simeq \mathrm{NS}(S'_{k^s})_\Q$. This means
	that $M$ induces an isomorphism
	$\h^2_{\mathrm{alg}}(S) \stackrel{\sim}{\longrightarrow}
	\h^2_{\mathrm{alg}}(S')$
	with inverse given by its transpose.
	It follows that $\Gamma_{\mathrm{alg}} := o_S \times S' + M + S\times o_{S'}$
	induces an
	isomorphism $\h_{\mathrm{alg}}(S)
	\stackrel{\sim}{\longrightarrow} \h_{\mathrm{alg}}(S')$ with inverse
	${}^t\Gamma_{\mathrm{alg}}$. In addition, we have $(\Gamma_{\mathrm{alg}})_*
	o_S = o_{S'}$.
	
	Second, recall from \S \ref{subsec:weight} that $v_2(\mathcal{E})$
	induces an isomorphism $\h^2_{\mathrm{tr}}(S) \stackrel{\sim}{\longrightarrow}
	\h^2_{\mathrm{tr}}(S')$ with inverse induced by $v_2(\mathcal{E}^\vee\otimes
	p^*\omega_S)$.
	Since K3 surfaces have trivial first Chern class and trivial
	canonical bundle, it follows that the inverse of $v_2(\mathcal{E})$ is in fact
	its transpose.
	In other words,	 $\Gamma_{\mathrm{tr}}:= \pi^2_{\mathrm{tr},S'} \circ
	v_2(\mathcal{E}) \circ
	\pi^2_{\mathrm{tr},S}$ induces an isomorphism of Chow motives
	$\h^2_{\mathrm{tr}}(S)
	\stackrel{\sim}{\longrightarrow}\h^2_{\mathrm{tr}}(S')$ with inverse its
	transpose.
	Finally, we do have $(\pi^2_{\mathrm{tr}})_*o_S =0$
	because of
	the orthogonality of $\pi^2_{\mathrm{tr},S}$ with $\pi^4_S$. 
	
	The required correspondences $\Gamma_{\mathrm{alg}}$ and $\Gamma_{\mathrm{tr}}$
	have thus been constructed and this concludes the proof of
	Theorem~\ref{thm:main}.
	\qed
	
	\subsection{Proof of Corollary \ref{cor:Powers}} 
	Let $S$ and $S'$ be two twisted derived equivalent K3 surfaces. Then due to
	Theorem \ref{thm:main} their motives are isomorphic as Frobenius algebra
	objects. As is
	explained in \S \ref{SS:alg},
	isomorphisms of Frobenius algebra objects behave well with respect to (tensor)
	products, hence it suffices to see that for any $n\in \Z_{>0}$ the Hilbert
	schemes of length-$n$ subschemes $\Hilb^{n}(S)$ and $\Hilb^{n}(S')$ have isomorphic Chow
	motives as Frobenius algebra objects. To this end, we use the result of
	Fu--Tian
	\cite{FT} that describes the algebra object $\h(\Hilb^{n}(S))$ in terms of the
	algebra objects
	$\h(S^{m})$ for $m\leq n$, together with some explicit combinatorial rules.
	More precisely, by \cite[Theorem 1.6 and Remark 1.7]{FT}, for a K3 surface $S$,
	we have an isomorphism of algebra objects\,:
	\begin{equation}\label{eqn:IsoFT}
	\phi: \h(\Hilb^{n}(S))\simeq \left(\bigoplus_{g\in
		\mathfrak{S}_{n}}\h\left(S^{O(g)}\right), \star_{\operatorname{orb},
		\operatorname{dt}}\right)^{\mathfrak{S}_{n}},
	\end{equation}
	where $\mathfrak{S}_{n}$ is the symmetric group acting naturally on $S^{n}$\,;
	for a permutation $g$, $O(g)$ is its set of orbits in $\{1, \cdots, n\}$,
	$S^{O(g)}$ is canonically identified with the fixed locus $(S^{n})^{g}$, and
	finally $\star_{\operatorname{orb}, \operatorname{dt}}$ is the orbifold product
	with discrete torsion (see \cite{ftv, FT}) defined as follows (let us omit the
	Tate twists for ease of notation)\,: it is  compatible with the
	$\mathfrak{S}_{n}$-grading, and for any $g, h \in \mathfrak{S}_{n}$,
	$\h\left(S^{O(g)}\right)\otimes \h\left(S^{O(h)}\right)\to
	\h\left(S^{O(g,h)}\right)$ is given by the pushforward via the diagonal
	inclusion $S^{O(g,h)}\hookrightarrow S^{O(g)}\times S^{O(h)}\times S^{O(gh)}$ 
	of the cycle $\epsilon(g,h)c_{g,h}\in \CH(S^{O(g, h)})$, by \cite[Lemma
	9.3]{FT}\,:
	\begin{equation}\label{eqn:ObsCl}
	c_{g,h}:=
	\begin{cases}
	0, &  \text{ if } \exists\,  t\in O(g,h) \text{ with } d_{g,h}(t)\geq 2\, ;\\
	\prod_{t\in I}\left(24\pr_{t}^{*}(o_{S})\right), & \text{if } \forall t\in
	O(g,h) \text{ has } d_{g,h}(t)=0 \text{ or } 1,
	\end{cases}
	\end{equation}
	where $\epsilon(g,h):=(-1)^{\frac{n-|O(g)|-|O(h)|+|O(gh)|}{2}}$, $O(g, h)$ is
	the set of orbits in $\{1, \cdots, n\}$ under the subgroup generated by $g$ and
	$h$\,; for any orbit $t\in O(g, h)$,
	$$d(g,h)(t):=\frac{2+|t|-|t/g|-|t/h|-|t/gh|}{2}$$ is the \emph{graph defect}
	function \cite[Lemma 9.1]{FT} and $I:=\{t\in O(g,h)~\vert~ d_{g,h}(t)=1\}$ is
	the subset of orbits with graph defect 1.
	
	As our isomorphism of algebra objects $\Gamma: \h(S)\to \h(S')$ satisfies 
	$\Gamma_{*}(o_{S})=o_{S'}$, it is now clear from the above precise description
	that the right-hand side of \eqref{eqn:IsoFT} for $S$ and for $S'$ are
	isomorphic algebra objects, and the isomorphism can be chosen orthogonal. As
	the
	morphism $\phi$ in \eqref{eqn:IsoFT} satisfies	 $\phi^{-1}={}^{t}\phi$, we have
	$\h(\Hilb^{n}(S))$ and $\h(\Hilb^{n}(S'))$ are
	isomorphic Frobenius algebra objects. This completes the proof.\qed
	
	\begin{rmk}[Chow rings \textit{vs.}~algebra objects]\label{rmk:ChowRingIso}
		It turns out that we do not need Theorem \ref{thm:main} to show that two
		twisted derived equivalent K3 surfaces have isomorphic Chow rings. Indeed, 
		Huybrechts' result \cite{huybrechts-derivedeq} (generalized to the twisted
		case
		in \cite{huybrechts-isogenous}) provides a correspondence $\Gamma\in
		\CH^{2}(S\times S')$ that induces an isomorphism of graded $\Q$-vector spaces
		$\Gamma_{*}: \CH^{*}(S)\stackrel{\sim}{\longrightarrow}\CH^{*}(S')$ with the
		extra property of being isometric on the N\'eron--Severi spaces
		$\CH^{1}(S)\stackrel{\sim}{\longrightarrow} \CH^{1}(S')$.  Now thanks to the
		theorem of Beauville--Voisin \cite{bv} saying that the image of the
		intersection
		product of two divisors on a K3 surface is of dimension 1,  this already
		implies
		that $\Gamma_{*}$ is actually an isomorphism of graded $\Q$-algebras. 
		
		In contrast, in the situation of Corollary \ref{cor:Powers}, a derived
		equivalence between $D^{b}(S)$ and $D^{b}(S')$ does give rise to a derived
		equivalence between their powers and Hilbert schemes, thanks to
		Bridgeland--King--Reid \cite{BKR} and Haiman \cite{Haiman}. However, it is not
		at all clear for the authors how to produce an isomorphism of the Chow rings
		(or
		even the rational cohomology rings) of two derived equivalent holomorphic
		symplectic varieties starting from the Fourier--Mukai kernel\,; see Conjecture
		\ref{conj:HK}.
	\end{rmk}

	\subsection{Proof of Corollary \ref{cor:torelli}} \label{S:coro}
	The equivalence of ${(i)}$ and ${(ii)}$ is due to Huybrechts \cite[Corollary
	1.4]{huybrechts-isogenous}, while the implication $(ii) \Rightarrow (iii)$ is
	Theorem \ref{thm:main}. We now prove
	the implication $(iii) \Rightarrow (i)$. Suppose that $\Gamma : \h(S)
	\rightarrow \h(S')$ is an isomorphism that preserves the algebra structures.
	Let $c$ be the rational number such that $\Gamma^{*}[S']=c\,[S]$, or
	equivalently such that $\Gamma^{-1}=\frac{1}{c}{}^{t}\Gamma$ by
	Lemma~\ref{lem:formal}\,; then
	the following diagram is commutative\,:
	
	\begin{equation}\label{eq:diagmultcoho}
	\xymatrix{\HH^2(S)\otimes \HH^2(S) \ar[rr]^{\quad \cup} \ar[d]^{(\Gamma\otimes
			\Gamma)_*} && \HH^4(S) \ar[d]^{\Gamma_*} \ar[rr]^{\deg} && \Q \ar[d]^{\cdot
			c}\\
		\HH^2(S')\otimes \HH^2(S') \ar[rr]^{\quad \cup} && \HH^4(S') \ar[rr]^{\deg} &&
		\Q.
	}
	\end{equation} 
	The commutativity of the left-hand square of \eqref{eq:diagmultcoho} is
	implied directly by the assumption that $\Gamma$ preserves the algebra
	structures, while the commutativity of the right-hand square follows from the
	Poincar\'e dual of the identity $\Gamma^{*}[S']=c\,[S]$. If  in addition
	$\Gamma$ preserves the Frobenius algebra structure, then $c=1$ by
	Proposition~\ref{prop:RigidAlgIso}. This means that $S$ and $S'$ are isogenous.
	\qed

	\subsection{A motivic global Torelli theorem} The aim of this section is to
	show that  Lemma~\ref{lem:formal} directly allows to upgrade motivically the
	global Torelli theorem, without utilizing the decomposition of the diagonal of
	Beauville--Voisin (Theorem~\ref{thm:bv}). We denote $\h(X)_\Z$ the Chow motive
	of $X$ with integral coefficients.
	
	\begin{thm}[Motivic global Torelli theorem for K3
		surfaces]\label{thm:motglobtor} 
		Let $S$ and $S'$ be two complex projective K3 surfaces. The following
		statements are
		equivalent\,:
		\begin{enumerate}[(i)]
			\item $S$ and $S'$ are isomorphic\,;
			\item $\HH^2(S,\Z)$ and $\HH^2(S',\Z)$ are Hodge isometric\,;
			\item $\h(S)_\Z$ and $\h(S')_\Z$ are isomorphic as algebra objects.
		\end{enumerate}	
	\end{thm}
	\begin{proof}
		The equivalence of items $(i)$ and $(ii)$ is the global Torelli theorem. The
		implication $(i)\Rightarrow (iii)$ is obvious. It remains to check that $(iii)
		\Rightarrow (ii)$. Once it is observed that Lemma~\ref{lem:formal}
		holds with integral coefficients, we obtain the following commutative diagram
		(with $c\in \Z$), which is similar to \eqref{eq:diagmultcoho} in the proof of
		Corollary~\ref{cor:torelli} 
		\begin{equation*}
		\xymatrix{\HH^2(S, \Z)\otimes \HH^2(S, \Z) \ar[rr]^{\quad \cup}
			\ar[d]_{\simeq}^{(\Gamma\otimes
				\Gamma)_*} && \HH^4(S, \Z) \ar[d]_{\simeq}^{\Gamma_*}
			\ar[rr]_{\simeq}^{\deg}
			&& \Z \ar[d]_{\simeq}^{\cdot c}\\
			\HH^2(S', \Z)\otimes \HH^2(S', \Z) \ar[rr]^{\quad \cup} && \HH^4(S', \Z)
			\ar[rr]_{\simeq}^{\deg} &&
			\Z.
		}
		\end{equation*} 		
		Therefore, there is an isometry of lattices between $\HH^{2}(S, \Z)\otimes
		\langle c\rangle$ and $\HH^{2}(S', \Z)$, which implies that $c=1$.
	\end{proof}

	\section{Beyond K3 surfaces}\label{sect:MultOrlov}
	From now on, the base field will be the field of complex numbers.
	Orlov's conjecture \ref{conj:orlov} predicts that the Chow motives of
	two derived equivalent smooth projective varieties are isomorphic. Motivated by
	Theorem \ref{thm:main}, we raised the following question in the introduction\,:
	
	\begin{ques}\label{ques:DeriveFrob}
		When can we expect more strongly that a derived equivalence between two smooth
		projective varieties implies an
		isomorphism between their rational Chow motives as \emph{Frobenius algebra
			objects}\,? 
	\end{ques}
	We make some remarks and speculations on this question in this section. 
	
	\begin{rmk}
		By Bondal--Orlov \cite{bo}, two derived equivalent smooth projective varieties
		that are either Fano or with ample canonical bundle are isomorphic\,; in
		particular, their motives are isomorphic as Frobenius algebra objects.
		Similarly, Question \ref{ques:DeriveFrob} also has a positive answer for
		curves,
		as they do not have
		non-isomorphic Fourier--Mukai partners \cite[Corollary
		5.46]{HuybrechtsFMBook}.
	\end{rmk}
	
	In general, one cannot expect in general a positive answer to Question
	\ref{ques:DeriveFrob}. In fact, if $\h(X)$ and $\h(Y)$ are isomorphic as
	Frobenius algebra objects then by applying the Betti realization functor, their
	cohomology are isomorphic as Frobenius algebras, that is, due to
	Proposition~\ref{prop:RigidAlgIso}, there is a (graded)
	isomorphism of $\Q$-algebras $\HH^{*}(X, \Q)\to\HH^{*}(Y, \Q)$ sending
	the class of a point on~$X$ to the class of a point on $Y$. However, as we
	will see below, this is not the case in general for derived equivalent
	varieties.  
	
	\subsection{Calabi--Yau varieties}
	\begin{ex}\label{ex:BC}
		Borisov and C\u ald\u araru \cite{bc} constructed derived equivalent
		(but non-birational) Calabi--Yau threefolds $X$ and $Y$ with the following
		properties\,:  $\Pic(X)
		=\Z H_X$ with $\deg(H_X^3)
		= 14$ and $\Pic(Y) = \Z H_Y$ with $\deg(H_Y^3) = 42$\,; hence there is no
		graded $\Q$-algebra isomorphism between $\HH^*(X,\Q)$ and $\HH^*(Y,\Q)$ that
		respects the point class. Therefore, $\h(X)$ and $\h(Y)$ are not isomorphic as
		Frobenius algebra objects. Nevertheless, thanks to the following proposition,
		$\HH^*(X,\Q)$ and $\HH^*(Y,\Q)$ are Hodge isomorphic as graded $\Q$-algebras
		and
		also as graded Frobenius algebras after extending the coefficients to $\mathbb
		R$.\footnote{We are not aware of any examples of derived equivalent smooth projective varieties with
			non-isomorphic cohomology as $\Q$-algebras or as $\mathbb{R}$-Frobenius
			algebras.}
	\end{ex}
	
	\begin{prop}
		Let $X$ and $Y$ be two derived equivalent Calabi--Yau varieties of dimension
		$d\geq 3$. Suppose their Hodge numbers satisfy 
		\begin{itemize}
			\item  $h^{p,q}=0$ for all $p\neq q$ and $p+q\neq d$;
			\item $h^{p,p}=1$ for all $2p\neq d$ and $0\leq p \leq d$.
		\end{itemize} 
		Then\,:
		\begin{enumerate}[$(i)$]
			\item There is a (graded) real Frobenius algebra isomorphism between
			$\HH^{*}(X, \mathbb{R})$ and $\HH^{*}(Y, \mathbb{R})$ preserving the real
			Hodge
			structures.	
			\item If $d$ is odd or $d$ is even and $s:= \frac{\deg(Y)}{\deg(X)}$ is a
			square in
			$\Q$, then $\HH^{*}(X, \Q)$ and $\HH^{*}(Y, \Q)$ are isomorphic as graded
			$\Q$-Hodge algebras. Here the degree is the top self-intersection number of
			the
			ample generator of the Picard group.
		\end{enumerate}
	\end{prop}
	\begin{proof}
		We first prove $(ii)$.
		Let $\mathcal{E}$ be the Fourier--Mukai kernel of the equivalence from
		$\D^{b}(X)$ to $\D^{b}(Y)$. By \cite[Proposition 5.44]{HuybrechtsBook}, the
		correspondence given by the Mukai vector $v(\mathcal{E})\in \CH^{*}(X\times
		Y)$
		induces a $\Z/2\Z$-graded Hodge isometry $$\Phi^{\HH}_{\mathcal{E}}:
		\HH^{*}(X,
		\Q)\stackrel{\sim}{\longrightarrow} \HH^{*}(Y, \Q),$$
		where both sides are equipped with the Mukai pairing. Note that as the
		varieties are Calabi--Yau, the Mukai pairing is simply given by the
		intersection pairing with
		some extra sign changes (\cite[Definition 5.42]{HuybrechtsBook}). The
		\emph{transcendental cohomology} denoted by $\HH^{*}_{\operatorname{tr}}(-,
		\Q)$
		is defined to be the orthogonal complement of the space of Hodge classes\,; it is obviously
		preserved by $\Phi^{\HH}_{\mathcal{E}}$. Thanks to our assumption on the Hodge
		numbers, the transcendental cohomology is concentrated in degree $d$.
		Therefore
		by restricting $\Phi_{\mathcal{E}}^{\HH}$, we get a Hodge isometry 
		$$\phi_{\operatorname{tr}}: \HH^{d}_{\operatorname{tr}}(X,
		\Q)\stackrel{\sim}{\longrightarrow}\HH^{d}_{\operatorname{tr}}(Y, \Q).$$
		On the other hand, if $d$ is even,
		$\Phi_{\mathcal{E}}^{\HH}$ also provides an isometry between the
		subalgebras of Hodge classes $\Hdg^{*}_{\Q}(X)$ and $\Hdg^{*}_{\Q}(Y)$. Since
		the quadratic space $\HH^{0}\oplus\cdots\oplus \HH^{d-2}\oplus \HH^{d+2}\oplus
		\cdots\oplus \HH^{2d}$ equipped with the restriction of the Mukai pairing is
		isometric
		to $U^{\frac{d}{2}}\otimes \Q$ for both $X$ and $Y$, the quadratic spaces
		$\Hdg^{d}_{\Q}(X)$ and $\Hdg^{d}_{\Q}(Y)$ are isometric by Witt cancellation
		theorem. Due to the
		assumption that $s:= \frac{\deg(Y)}{\deg(X)}$ is a square and to Witt's
		theorem, we
		have an isometry $$\phi_{\Hdg}:
		\Hdg^{d}_{\Q}(X)\left(s\right) \longrightarrow \Hdg^{d}_{\Q}(Y)$$
		that sends $H_{X}^{\frac{d}{2}}$ to $H_{Y}^{\frac{d}{2}}$, where
		$H_{X}$ and $H_{Y}$ denote the ample generators of $\Pic(X)$ and $\Pic(Y)$,
		respectively.
		
		Let us now try to define a graded Hodge algebra isomorphism $\psi: \HH^{*}(X,
		\Q) \to \HH^{*}(Y, \Q)$. Consider the following formulas with the numbers $a,
		b$
		to be determined later\,:
		\begin{itemize}
			\item $H^{i}_{X}\mapsto a^{i}\cdot H^{i}_{Y}$ for all $0\leq i\leq d$ and
			consequently $[\pt_{X}]\mapsto a^{d}\, s\cdot[\pt_{Y}]$, where
			$[\pt]$ is the class of a point\,;
			\item $a^{\frac{d}{2}}\cdot\phi_{\Hdg}: \Hdg_{\Q}^{d}(X)\to
			\Hdg^{d}_{\Q}(Y)$;
			\item $b\cdot\phi_{\tr}:\HH^{d}_{\tr}(X, \Q)\to \HH^{d}_{\tr}(Y, \Q)$.
		\end{itemize}
		These formulas define an algebra isomorphism if and only if
		$b^{2}=a^{d}\, s$. This equation has non-zero rational
		solutions when $d$ is odd or $d$ is even and $s$ is a
		square in $\Q$. Item $(ii)$ is therefore  proved.
		
		The proof of $(i)$ goes similarly as for $(ii)$ by replacing $\Q$ by
		$\mathbb{R}$. Notice that the analogous assumption that
		$s$ is a square in $\mathbb{R}$ is automatically
		satisfied. So it is enough to see that there are always non-zero real
		solutions
		to the equation $b^{2}=a^{d}\, s=1$, where the last equality
		reflects the Frobenius condition.
	\end{proof}
	
\begin{ex}
	let $V$ and $W$ be vector spaces of dimension 5 and 10, respectively. Given two generic isomorphisms $\phi_i\colon \bigwedge^2V\xrightarrow{\simeq} W$, $i\in \{1, 2\}$, one obtains two embeddings $\operatorname{Gr}(2, V)\hookrightarrow \mathbb{P}(W)$, whose intersection is a Calabi--Yau 3-fold $X$ of Picard number 1. Using the inverses of the dual isomorphisms $\phi^\vee_i\colon W^\vee  \xrightarrow{\simeq} \bigwedge^2V^\vee$, one defines similarly another Calabi--Yau 3-fold $Y$ in~$\mathbb{P}(W^\vee)$.
In~\cite{bcp}, Borisov, C\u ald\u araru and Perry showed that $X$ and $Y$ are derived equivalent (but non-birational). As $X$ and $Y$ have the same degree (i.e.~$s=1$), the argument in the previous proposition shows that there is a (graded)  Frobenius algebra isomorphism between
$\HH^{*}(X, \mathbb{Q})$ and $\HH^{*}(Y, \mathbb{Q})$ preserving the Hodge structures.	
\end{ex}
	\subsection{Abelian varieties}
	
	\begin{prop}[Isogenous abelian varieties]\label{prop:AV}
		Let $A$ and $B$ be two isogenous abelian varieties of dimension $g$. Then 
		\begin{enumerate}[$(i)$]
			\item $\h(A)$ and $\h(B)$ are isomorphic as algebra objects.
			\item The following conditions are equivalent:
			\begin{enumerate}
				\item  There is an isomorphism of Frobenius algebra objects between $\h(A)$
				and
				$\h(B)$.
				\item There is a graded Hodge isomorphism of Frobenius algebras between
				$\HH^{*}(A, \Q)$ and $\HH^{*}(B, \Q)$.
				\item There exists an isogeny of degree $m^{2g}$ between $A$ and $B$  for
				some
				$m\in \Z_{>0}$.
			\end{enumerate} 
			In the case that these equivalent conditions hold, the isomorphism in $(a)$,
			denoted by $\Gamma: \h(A)\to \h(B)$,  can be chosen to respect moreover the
			motivic decomposition of Deninger--Murre \cite{dm} in the sense that 
			$\Gamma\circ \pi_{A}^{i}=\pi_{B}^{i}\circ \Gamma$ for any $i$,
			where the $\pi^{i}$'s are the projectors corresponding to the decomposition.
			\item $\h(A)_{\mathbb{R}}$ and $\h(B)_{\mathbb{R}}$ are isomorphic as
			Frobenius
			algebra objects in the category of Chow motives with real coefficients. 
		\end{enumerate}
	\end{prop}
	\begin{proof}
		$(i)$ Consider any isogeny $f: B\to A$. Then $f^{*}: \h(A)\to \h(B)$ is an
		isomorphism of algebra objects with inverse given by
		$\frac{1}{\deg(f)}f_{*}$.\\
		$(ii)$ The implication $(a)\Longrightarrow (b)$ is obtained by applying the
		realization functor.\\
		$(b)\Longrightarrow (c)$. Let $\gamma: \HH^{*}(A,
		\Q)\stackrel{\sim}{\longrightarrow}\HH^{*}(B, \Q)$ be a Frobenius algebra
		isomorphism preserving the Hodge structures, and let $\gamma_{i}: \HH^{i}(A,
		\Q)\to \HH^{i}(B, \Q)$ be its $i$-th component, for all $0\leq i\leq 2g$.
		There
		exist a rational number $\lambda$ and an isogeny $f: B\to A$, such that
		$\gamma_{1}: \HH^{1}(A, \Q)\to \HH^{1}(B, \Q)$ is equal to 
		$\frac{1}{\lambda}f^{*}|_{\HH^{1}}$. As $\HH^{*}(A, \Q)\cong
		\bigwedge^{\bullet}
		\HH^{1}(A, \Q)$ as algebras and similarly for $B$, $\gamma$ is in fact
		determined by $\gamma_{1}$ in the following way\,: for any $i$, 
		$\gamma_{i}=\wedge^{i}\gamma_{1}=\frac{1}{\lambda^{i}}f^{*}|_{\HH^{i}}$. We
		compute that $$\id={}^{t}\gamma\circ
		\gamma=\left(\sum_{i}\frac{1}{\lambda^{i}}
		f_{*}|_{\HH^{2g-i}}\right)\circ
		\left(\sum_{i}\frac{1}{\lambda^{i}}f^{*}|_{\HH^{i}}\right)=\frac{1}{\lambda^{2g}}\deg(f)\cdot\id.$$
		This  yields that the isogeny $f$ is of degree $\lambda^{2g}$.\\
		$(c)\Longrightarrow (a)$	
		If there is an isogeny $f: B\to A$ of degree $m^{2g}$, then for any $0\leq
		i\leq 2g$ consider the morphism $\Gamma_{i}:=\frac{1}{m^{i}}\pi^{i}_{B}\circ
		f^{*}\circ \pi^{i}_{A}=\frac{1}{m^{i}} f^{*}\circ \pi^{i}_{A}$ from
		$\h^{i}(A)$
		to $\h^{i}(B)$, which is an isomorphism with inverse
		$\Gamma_{i}^{-1}=\frac{1}{m^{2g-i}}\pi^{i}_{A}\circ f_{*}$. Here we use the
		motivic decomposition of
		Deninger--Murre \cite{dm} for abelian varieties
		$\h(A)=\oplus_{i=0}^{2g}\h^{i}(A)$, and $\pi^{i}$ is the projector
		corresponding
		to $\h^{i}$. 
		One readily checks that $\Gamma:=\sum_{i} \Gamma_{i}: \h(A)\to \h(B)$ is an
		isomorphism of algebra objects. Moreover, as $\pi^{i}={}^{t}\pi^{2g-i}$ for
		all
		$i$, we have that $\Gamma_{i}^{-1}={}^{t}\Gamma_{2g-i}$, hence
		$\Gamma^{-1}={}^{t}\Gamma$, that is, $\Gamma$ respects the Frobenius
		structures.
		Notice that by construction, $\Gamma$ respects the
		decomposition of  Deninger--Murre.\\	
		The proof of $(iii)$ is similar to the last part of the proof of $(ii)$.
		One only needs to notice that there is no obstruction to taking the $2g$-th
		root of a
		positive number in $\mathbb{R}$. 
	\end{proof}
	
	As a consequence, given two derived equivalent abelian varieties, in general
	there is no isomorphism of Frobenius algebra objects between their Chow motives
	(or their cohomology). Indeed, by Proposition \ref{prop:AV}$(ii)$, the motives
	of two derived equivalent abelian varieties that cannot be related by an isogeny
	of degree
	the $2g$-th power of some positive integer are not isomorphic as Frobenius
	algebra objects. For instance, if one considers an abelian variety $A$ with
	N\'eron--Severi group generated by one ample line bundle $L$, then any isogeny
	between $A$ and $A^{\vee}$ is of degree $\chi(L)^{2}m^{4g}$ for some $m\in
	\Z_{>0}$. But in general, $\chi(L)$ is not a $g$-th power in $\Z$. On the other
	hand, $A$ and $A^{\vee}$ are always derived equivalent by Mukai's classical
	result \cite{MukaiAb}.

	\subsection{Hyper-K\"ahler varieties}
	
	One particularly interesting class of varieties for which we expect a positive
	answer consists of (projective) hyper-K\"ahler varieties\,; these  constitute
	higher-dimensional generalizations of K3 surfaces.  Note that by
	Huybrechts--Nieper-Wi\ss kirchen \cite{HNW}, any Fourier--Mukai partner of a
	hyper-K\"ahler variety remains  hyper-K\"ahler. 
	
	\begin{conj}\label{conj:HK}
		Let $X$ and $Y$ be two projective hyper-K\"ahler varieties. If there is an
		exact
		equivalence between triangulated categories $\operatorname{D}^b(X)$ and
		$\operatorname{D}^b(Y)$, then there exists an isomorphism of Chow motives
		$\h(X)$ and $\h(Y)$, as Frobenius algebra objects in the categories of Chow
		motives. In
		particular, their Chow rings as well as cohomology rings are isomorphic.
	\end{conj}
	
	The following result is known to the experts\,; it answers the last part of
	Conjecture~\ref{conj:HK} for cohomology with complex coefficients.
	\begin{prop}
		Let $X$ and $Y$ be two derived equivalent projective hyper-K\"ahler varieties.
		Then their cohomology rings with complex coefficients are isomorphic as
		$\C$-algebras.
	\end{prop}
	\begin{proof}\footnote{It is not clear to the authors whether the isomorphism
			constructed in the proof preserves the Frobenius structure.}
		As any exact equivalence $\operatorname{D}^{b}(X)\simeq
		\operatorname{D}^{b}(Y)$
		is given by Fourier--Mukai kernel, it lifts naturally to an equivalence of
		differential graded categories. Therefore we have an isomorphism of graded
		$\C$-algebras between their Hochschild cohomology:
		$$\HHH^{*}(X)\simeq \HHH^{*}(Y).$$ 
		By a result of Huybrechts--Nieper-Wisskirchen \cite{HNW} (see Calaque--Van den Bergh \cite{CVdB} in a greater generality), which was also previously
		announced by Kontsevich, the Hochschild--Kostant--Rosenberg isomorphism
		twisted
		by the square root of the Todd genus gives rise to an isomorphism of
		$\C$-algebras
		$$\HHH^{*}(X)\simeq \bigoplus_{i+j=*}\HH^{i}(X, \bigwedge^{j}T_{X}).$$ Now the
		symplectic forms on $X$ induce an isomorphism between $T_{X}$ and
		$\Omega_{X}$,
		which yields isomorphisms of $\C$-algebras:   $$\bigoplus_{i+j=*}\HH^{i}(X,
		\bigwedge^{j}T_{X})\simeq \bigoplus_{i+j=*}\HH^{i}(X, \Omega^{j}_{X})\simeq
		\HH^{*}(X, \C).$$
		We can conclude by combining these isomorphisms.
	\end{proof}
	
	There are not so many known examples of derived equivalent hyper-K\"ahler
	varieties. Let us test Conjecture \ref{conj:HK} for the available ones.
	\begin{ex}
		Let $S$ and $S'$ be two derived equivalent K3 surfaces. Then for any $n\in
		\mathbb N^{*}$, the $n$-th Hilbert schemes $\Hilb^{n}(S)$ and $\Hilb^{n}(S')$
		are
		derived
		equivalent. Indeed, by combining the results of Bridgeland--King--Reid
		\cite{BKR} and Haiman \cite{Haiman}, we have exact linear equivalences of
		triangulated categories\,:
		$$\D^{b}(\Hilb^{n}(S))\simeq
		\D^{b}(\mathfrak{S}_{n}\!-\!\operatorname{Hilb}(S^{n}))\simeq
		\D^{b}_{\mathfrak
			S_{n}}(S^{n}),$$
		and similarly for $S'$\,; the Fourier--Mukai kernel
		$\mathcal{E}\boxtimes\cdots
		\boxtimes \mathcal{E}$ induces an equivalence $$\D^{b}_{\mathfrak
			S_{n}}(S^{n})\simeq \D^{b}_{\mathfrak S_{n}}(S'^{n}),$$ where $\mathcal E\in
		\D^{b}(S\times S')$ is the original Fourier--Mukai kernel inducing the
		equivalence between $\D^{b}(S)$ and $\D^{b}(S')$. We showed in Corollary
		\ref{cor:Powers} that $\h(\Hilb^{n}(S))$
		and
		$\h(\Hilb^{n}(S'))$ are isomorphic as  Frobenius algebra objects.
	\end{ex}
	
	\begin{ex}
		Conjecturally two birationally equivalent hyper-K\"ahler varieties are derived
		equivalent \cite[Conjecture 6.24]{HuybrechtsFMBook}. Thanks to the result of
		Rie\ss~\cite{Riess}, or rather its proof, we know that birational
		hyper-K\"ahler varieties have
		isomorphic Chow motives as Frobenius algebra objects, hence compatible with
		Conjecture
		\ref{conj:HK}.
		There are by now some cases where the derived equivalence is known. The
		easiest
		example might be the so-called Mukai flop. Another instance of interest is as
		follows\,: given a projective K3 surface $S$ and a Mukai vector $v$, when the
		stability condition $\sigma$ varies in the chambers of the distinguished
		component $\operatorname{Stab}^{*}(S)$ of the manifold of stability conditions
		on $\D^{b}(S)$, the moduli spaces $M_{\sigma}(v)$ of $\sigma$-stable objects
		are
		all birational to each other, and their derived equivalence has been announced
		by Halpern-Leistner in~\cite{DHL}.
		
	\end{ex}

	\begin{ex}
		If one is willing to enlarge a bit the category of hyper-K\"ahler varieties to
		that of hyper-K\"ahler \emph{orbifolds}\footnote{Here by an \emph{orbifold} we
			mean a smooth proper Deligne--Mumford stack with trivial generic
			stabilizer.},
		Conjecture \ref{conj:HK} is closely related to the so-called \emph{motivic
			hyper-K\"ahler resolution conjecture} investigated in \cite{ftv} and
		\cite{FT}.
		Indeed, let $M$ be a projective holomorphic symplectic variety endowed with a
		faithful
		action of a finite group $G$ by symplectic automorphisms. The quotient stack
		$[M/G]$ is a hyper-K\"ahler (or rather symplectic) orbifold. If the main
		component of the $G$-invariant
		Hilbert scheme $X:=G\!-\!\Hilb(M)$ is a symplectic (or equivalently crepant)
		resolution of the singular variety $M/G$, then by Bridgeland--King--Reid
		\cite[Corollary~1.3]{BKR} there is an equivalence of derived categories
		$\D^{b}(X)\simeq \D^{b}([M/G])$. On the other hand, the motivic hyper-K\"ahler
		resolution conjecture \cite{ftv} predicts that the orbifold motive of $[M/G]$
		endowed with the \emph{orbifold product} is
		isomorphic to the motive of $X$ as algebra objects. In this sense, forgetting
		the Frobenius structure, we can obtain some evidences for the orbifold
		analogue of Conjecture \ref{conj:HK}\,: for example between a K3 orbifold and
		its minimal resolution by~\cite{FTsurface}, between
		$[\ker(A^{n+1}\xrightarrow{+} A)/\mathfrak{S}_{n}]$ and the $n$-th generalized
		Kummer variety associated to an abelian surface $A$ by \cite{ftv}, and between
		$[S^{n}/\mathfrak{S}_{n}]$ and the $n$-th Hilbert scheme of a K3 surface $S$ by
		\cite{FT}. 
		In
		fact, the authors suspect that the motivic hyper-K\"ahler resolution
		conjecture
		can be stated more strongly as an isomorphism of Frobenius algebra objects
		with
		\emph{complex} coefficients, and the proofs of our aforementioned results do
		confirm this stronger version. 
	\end{ex}

	\section{Chern classes of Fourier--Mukai equivalences between K3 surfaces}
	\label{sec:v2}
	
	The aim of this final section is to provide evidence for the fact that the
	Chern classes of Fourier--Mukai equivalences between two K3 surfaces $S$ and
	$S'$ define ``distinguished'' classes in the Chow ring of $S\times S'$, in the
	sense that they can be added to the Beauville--Voisin ring of $S\times S'$ and
	the resulting ring would still inject into cohomology \emph{via} the cycle class
	map.
	
	\subsection{The Beauville--Voisin ring, and generalizations}\label{subsec:mck}
	Let $S$ be a K3 surface and define its Beauville--Voisin ring $R^*(S)$ to be
	the subring of $\CH^*(S)$ generated by divisors and Chern classes of the
	tangent
	bundle. By Beauville--Voisin's Theorem~\ref{thm:bv}, this ring has the property
	that it injects into cohomology via the cycle class map.
	\medskip
	
	Let $\h(S) = \h^0(S)\oplus \h^2(S) \oplus \h^4(S)$ be the Chow--K\"unneth
	decomposition
	induced by $\pi^0_S = o_S\times S$, $\pi^4_S = S\times o_S$ and $\pi_S^2 =
	\Delta_S - \pi_S^0 - \pi_S^4$.
	In \cite[Proposition~8.14]{sv} it was observed that the decomposition of the
	small diagonal \eqref{eq:bv} is equivalent to the above Chow--K\"unneth
	decomposition
	being \emph{multiplicative}, meaning that the multiplication morphism $\h(S)
	\otimes \h(S) \to \h(S)$ is compatible with the grading given by the
	Chow--K\"unneth
	decomposition.
	\medskip
	
	The following (formal) facts about multiplicative Chow--K\"unneth
	decompositions will be
	used.
	Let $X$ and $Y$ be two smooth projective varieties, both having motive endowed
	with a multiplicative Chow--K\"unneth decomposition. Then
	\cite[Theorem~8.6]{sv} the
	\emph{product Chow--K\"unneth decomposition} $\h^n(X\times Y) =
	\bigoplus_{i+j=n}
	\h^i(X)\otimes \h^j(Y)$ is multiplicative. Moreover, if $p:X\times Y \to X$
	denotes the projection, then $p^* : \h(X) \to \h(X\times Y)$ is graded
	(\emph{i.e.} compatible with the Chow--K\"unneth decompositions) and $p_* :
	\h(X\times
	Y)
	\to \h(X)$ shifts the gradings by $-2\dim Y$. 
	
	A Chow--K\"unneth decomposition on the motive of $X$ induces a bigrading on the
	Chow
	groups of $X$ given by 
	$$\CH^i(X)_{(j)} := \CH^i(\h^{2i-j}(X)),$$ 
	which in case the Chow--K\"unneth decomposition is multiplicative satisfies
	$$\CH^i(X)_{(j)} \cdot \CH^{i'}(X)_{(j')} \subseteq \CH^{i+i'}(X)_{(j+j')} .$$
	Given smooth projective varieties endowed with multiplicative Chow--K\"unneth
	decompositions, the products of which are endowed with the product
	Chow--K\"unneth
	decompositions, we therefore see that $\CH^*(-)_{(0)}$ defines a subalgebra of
	$\CH^*(-)$ that is stable under pushforwards and pullbacks along projections,
	and stable under composition of correspondences belonging to $\CH^*(-\times
	-)_{(0)}$.\medskip
	
	Murre's conjecture~\ref{conj:murre}\eqref{B} and \eqref{D} imply that
	$\CH^i(X)_{(0)} := \CH^i(\h^{2i}(X))$ injects in cohomology with image the
	Hodge classes for any choice of
	Chow--K\"unneth decomposition. (This is known unconditionally in the case $i
	=0$ and $i = \dim
	X$.)
	In particular, in the above situation of smooth projective varieties endowed
	with multiplicative Chow--K\"unneth decompositions, it is expected that the
	subalgebra
	$\CH^*(-)_{(0)}$ injects into cohomology
	with image the Hodge classes. In that sense, $\CH^*(-)_{(0)}$ is a maximal
	subalgebra of $\CH^*(-)$ with the property that it injects into cohomology
	\emph{via} the cycle class map.
	
	\subsection{Adding the second Chern class of Fourier--Mukai equivalences to
		the BV ring}
	Recall the following theorem of Huybrechts \cite[Theorem~2]{huybrechts-JEMS}
	and
	Voisin \cite[Corollary~1.10]{voisin-ogrady}.
	
	\begin{thm}[Huybrechts, Voisin]\label{thm:hv}
		Let $\Phi_\mathcal{E} :  \operatorname{D}^b(S)
		\stackrel{\sim}{\longrightarrow}  \operatorname{D}^b(S')$ be an exact linear
		equivalence between K3 surfaces with Fourier--Mukai kernel $\mathcal{E} \in
		\operatorname{D}^b(S\times S')$. Then
		$v(\mathcal{E})$ preserves the Beauville--Voisin ring.
	\end{thm}
	
	In light of the discussion in \S\ref{subsec:mck},
	it is natural to ask whether a more general statement could be true, namely\,:
	
	\begin{ques} \label{ques:grade0}
		Let $\Phi_\mathcal{E} :  \operatorname{D}^b(S)
		\stackrel{\sim}{\longrightarrow}  \operatorname{D}^b(S')$ be an exact linear
		equivalence between K3 surfaces with Fourier--Mukai kernel $\mathcal{E} \in
		\operatorname{D}^b(S\times S')$. Then does
		$v(\mathcal{E})$ belong to $\CH^*(S\times S')_{(0)}$\,?
	\end{ques}
	
	For $i=0$ or $1$, the Mukai vectors $v_i(\mathcal E)$ obviously belong to
	$\CH^i(S\times S')_{(0)}$, since in those cases $\CH^i(S\times S')=
	\CH^i(S\times S')_{(0)}$.
	In the case of $v_2(\mathcal{E})$, this can be deduced from
	Theorem~\ref{thm:hv}\,:
	
	\begin{prop}\label{prop:v2}
		Let $\Phi_\mathcal{E} :  \operatorname{D}^b(S)
		\stackrel{\sim}{\longrightarrow}  \operatorname{D}^b(S')$ be an exact
		equivalence with Fourier--Mukai kernel $\mathcal{E} \in
		\operatorname{D}^b(S\times S')$. Then
		$v_2(\mathcal{E})$ belongs to $\CH^2(S\times S')_{(0)}$.
	\end{prop}
	\begin{proof}
		Since $\CH^2(S\times S') = \CH^2(S\times S')_{(0)} \oplus \CH^2(S\times
		S')_{(2)}$, 
		it is enough to show that $v_2(\mathcal{E})_{(2)} = 0$. Let $\gamma$ be any
		cycle in $\CH^2(S\times S')$. On the one hand, we have $\gamma_{(2)} =
		\pi^2_{\mathrm{tr},S'} \circ \gamma \circ \pi^4_S + \pi^0_{S'} \circ \gamma
		\circ
		\pi^2_{\mathrm{tr},S}$. On the other hand, we have $\gamma \circ \pi^4_S =
		(p')^*\gamma_*
		o_S$. Now setting $\gamma = v_2(\mathcal{E})$, Theorem~\ref{thm:hv} yields
		that $v_2(\mathcal E) \circ \pi^4_S$ is a multiple of $(p')^*o_{S'}$ and
		henceforth since $(\pi^2_{\mathrm{tr},S'})_*o_{S'} =0$ that
		$\pi^2_{\mathrm{tr},S'} \circ \gamma \circ \pi^4_S = 0$. Likewise, we have 
		$\pi^0_{S'} \circ v_2(\mathcal E)
		\circ
		\pi^2_{\mathrm{tr},S} = 0$, and the
		proposition is established.
	\end{proof}
	
	We deduce from \S \ref{subsec:mck} the following
	\begin{thm}\label{thm:product-c2}
		Let $\tilde{R}^*(S\times S')$ be the subring of $\CH^*(S\times S')$
		generated by divisors, $p^*c_2(S)$, $(p')^*c_2(S')$ and $c_2(\mathcal{E})$,
		where $\mathcal{E}$ runs through the objects in $\mathrm{D}^b(S\times S')$
		inducing exact linear equivalences
		$ \operatorname{D}^b(S)
		\stackrel{\sim}{\longrightarrow}  \operatorname{D}^b(S')$.
		The cycle class map $\tilde{R}^n(S\times S') \to \HH^{2n}(S\times S',\Q)$ is
		injective
		for $n=3,4$. In particular, if $\Phi_{\mathcal{E}_{1}}, \Phi_{\mathcal{E}_{2}}
		:  \operatorname{D}^b(S) \stackrel{\sim}{\longrightarrow} 
		\operatorname{D}^b(S')$ are two exact linear equivalences with Fourier--Mukai
		kernels
		$\mathcal{E}_{1}, \mathcal{E}_{2} \in \operatorname{D}^b(S\times S')$, then
		$c_2(\mathcal{E}_{1})\cdot c_2(\mathcal{E}_{2}) \in \Z[ o_S\times
		o_{S'}]$.\qed
	\end{thm}

	\subsection{Some speculations concerning Chern classes of twisted derived
		equivalent K3 surfaces}
	It is natural to ask whether Theorem~\ref{thm:hv} extends to derived
	equivalences between twisted K3 surfaces\,:
	
	\begin{ques} \label{ques:grade0twisted}
		Let $\Phi_\mathcal{E} :  \operatorname{D}^b(S,\alpha)
		\stackrel{\sim}{\longrightarrow}  \operatorname{D}^b(S',\alpha')$ be an exact
		equivalence between twisted K3 surfaces with Fourier--Mukai kernel
		$\mathcal{E} \in
		\operatorname{D}^b(S\times S',\alpha^{-1} \boxtimes \alpha')$. Then does
		$v(\mathcal{E})$ preserve the Beauville--Voisin ring\,? More generally, does
		$v(\mathcal{E})$ belong to $\CH^*(S\times S')_{(0)}$\,?
	\end{ques}
	
	We note that if $v(\mathcal{E})$ preserves the Beauville--Voisin ring, then the
	same argument as in the proof of Proposition~\ref{prop:v2} gives that
	$v_2(\mathcal{E})$ belongs to $\CH^2(S\times S')_{(0)}$.\medskip
	
	Let us now define $E^*(S\times S')$ to be the
	subalgebra of $\CH^*(S\times S')$ generated by divisors, $p^*c_2(S)$,
	$(p')^*c_2(S')$, and the Chern classes of $\mathcal{E}$, where $\mathcal{E}$
	runs through objects in 
	$\operatorname{D}^b(S\times S',\alpha^{-1} \boxtimes \alpha')$ inducing exact
	equivalences  $\Phi_\mathcal{E} :  \operatorname{D}^b(S,\alpha)
	\stackrel{\sim}{\longrightarrow}  \operatorname{D}^b(S',\alpha')$ for some
	Brauer classes $\alpha \in \mathrm{Br}(S)$ and $\alpha' \in \mathrm{Br}(S')$.
	We then define $\tilde{E}^*(S\times S')$ to be the
	subalgebra of $\CH^*(S\times S')$ generated by
	cycles of the form 
	$$\gamma_{n-1}\circ \cdots \circ \gamma_0,$$
	where $\gamma_i \in E^*(S_i\times S_{i+1})$ for all $i$ for some K3 surfaces
	$S=S_0,S_1,\ldots, S_{n}=S'$. 
	According to the discussion in \S\ref{subsec:mck},
	a positive answer to Question~\ref{ques:grade0twisted} would suggest that the
	following question should have a positive answer.
	
	\begin{ques}
		Does $\tilde{E}^*(S\times S')$ inject into cohomology via the cycle class
		map\,?
	\end{ques}
	
	In particular, if $\HH^2(S,\Q) \simeq \HH^2(S',\Q)$ is an isogeny, then the
	cycle class $v_2(\mathcal{E}_{n-1}) \circ \cdots \circ
	v_2(\mathcal{E}_0)$  inducing the isogeny between $T(S)_\Q :=
	\HH^*(\h^2_{\mathrm{tr}}(S))$ and $T(S')_\Q :=  \HH^*(\h^2_{\mathrm{tr}}(S))$
	(with $\mathcal{E}_0,\dots, \mathcal{E}_{n-1}$ as in \eqref{eq:tde}) should be
	canonically defined, \emph{i.e.} should not depend on the choice of twisted
	derived equivalence between $S$ and $S'$ as in \eqref{eq:tde} inducing the
	isogeny.

	%%%%%%%%%%%%%%%%%%%%
	%%%%%%%%%%%%%%%%%%%%

%	\pagebreak
	\appendix
	
	\section{Non-isogenous K3 surfaces with isomorphic Hodge
		structures}\label{sect:Appendix}
	
	Recall that two K3 surfaces $S$ and $S'$ are said to be \emph{isogenous} if
	their second rational cohomology groups are Hodge isometric, that is, 
	if	there exists an isomorphism of Hodge structures $f : \HH^2(S,\Q)
	\stackrel{\simeq}{\longrightarrow} \HH^2(S',\Q)$ making the following diagram
	commute\,: 	
	\begin{equation*}
	\xymatrix{\HH^2(S,\Q)\otimes \HH^2(S,\Q) \ar[rr]^{\qquad \cup} \ar[d]^{f\otimes
			f} && \HH^4(S,\Q)\ar[rr]^{\deg} && \Q \ar@{=}[d]\\
		\HH^2(S',\Q)\otimes \HH^2(S',\Q) \ar[rr]^{\qquad \cup} && \HH^4(S',\Q)
		\ar[rr]^{\deg} &&
		\Q.
	}
	\end{equation*}
	We provide in this appendix infinite families of pairwise non-isogenous K3
	surfaces  with isomorphic rational Hodge structures, that is, for any two K3
	surfaces $S$ and $S'$ belonging to  the same  family,  we have $\HH^{2}(S,
	\Q)\simeq \HH^{2}(S',
	\Q)$ as $\Q$-Hodge structures but there does not exist any isomorphism of
	$\Q$-Hodge structures $\HH^{2}(S, \Q)\to \HH^{2}(S', \Q)$ that is compatible
	with the intersection pairings given by $(\alpha,\beta) \mapsto \deg(\alpha\cup
	\beta)$. By Mukai \cite{Mukai}, such K3
	surfaces are not derived equivalent to each other, and actually not even
	through
	a chain of twisted derived equivalences (\emph{cf.} Huybrechts
	\cite[Theorem~0.1]{huybrechts-isogenous}).
	
	In fact, our families have a stronger property\,: any two K3 surfaces $S$ and
	$S'$ in the same family
	have Hodge-isomorphic rational cohomology \emph{algebras}, that is, for any two
	K3
	surfaces $S$ and $S'$ of the family,  we have a graded Hodge isomorphism $g :
	\HH^{*}(S, \Q)\stackrel{\simeq}{\longrightarrow} \HH^{*}(S',
	\Q)$  that respects the cup-product, \emph{i.e.} such that the following
	diagram commutes\,:
	\begin{equation*}
	\xymatrix{\HH^*(S,\Q)\otimes \HH^*(S,\Q) \ar[rr]^{\qquad \cup} \ar[d]^{g\otimes
			g} && \HH^*(S,\Q) \ar[d]^g \\
		\HH^*(S',\Q)\otimes \HH^*(S',\Q) \ar[rr]^{\qquad \cup} && \HH^*(S',\Q).
	}
	\end{equation*}
	We say that $\HH^{*}(S, \Q)$ and  $\HH^{*}(S',
	\Q)$ are \emph{Hodge algebra isomorphic}.

	\subsection{Motivation and statements} 
	The aim of this appendix is to show that, for K3 surfaces, the notion of
	\emph{isogeny}
	is strictly more restrictive than the notion of \emph{Hodge algebra
		isomorphic}. In \S \ref{SS:motivic} this is upgraded motivically\,: we show
	that, for Chow motives of K3 surfaces,
	the notion of being \emph{isomorphic as Frobenius algebra objects} is strictly
	more
	restrictive than the notion of being \emph{isomorphic as algebra objects},
	thereby
	justifying the somewhat technical condition $(iii)$ involving Frobenius
	algebra objects in the motivic Torelli statement of Corollary
	\ref{cor:torelli}.
	\medskip	
	
	First, we provide an infinite family of K3 surfaces whose
	rational
	cohomology rings are all Hodge algebra isomorphic (and, \emph{a fortiori},
	Hodge isomorphic), but they are pairwise non-isogenous.
	Precisely we have

	\begin{thm}\label{thm:non-isogenousK3}
		There exists an infinite family $\{S_i\}_{i \in \Z_{>0}}$  of pairwise
		non-isogenous K3 surfaces such
		that, for all $j, k \in \Z_{>0}$,  $\HH^*(S_j,\Q)$ and $\HH^*(S_k,\Q)$ are
		Hodge algebra isomorphic. Moreover, such a
		family can be chosen to consist of K3 surfaces of maximal Picard rank 20.
	\end{thm}

	We will present two proofs of this theorem, one in \S \ref{SS:2} 
	and the
	other in \S \ref{SS:1}. Furthermore, we will show in Theorem
	\ref{thm:IsomChowMot} that the Chow motives of K3 surfaces belonging to the
	family of K3 surfaces of Theorem~\ref{thm:non-isogenousK3} are all isomorphic
	as
	algebra objects.\medskip
	
	Second,	if the Picard number is maximal, a family of K3 surfaces as in Theorem
	\ref{thm:non-isogenousK3} must have non-isometric N\'eron--Severi spaces (Lemma
	\ref{lemma:rho=20}). In this sense, we can improve
	Theorem~\ref{thm:non-isogenousK3} in order to have in addition
	isometric $\Q$-quadratic forms on the N\'eron--Severi spaces\,:
	
	\begin{thm} \label{P:bothisonotisogenous}
		There exists an infinite family $\{S_i\}_{i \in \Z_{>0}}$  of pairwise
		non-isogenous K3 surfaces such
		that, for all $j\neq k \in \Z_{>0}$, we have
		\begin{itemize}
			\item $\HH^*(S_j,\Q)$ and $\HH^*(S_k,\Q)$ are Hodge algebra
			isomorphic.
			\item 	$\HH^2_{\operatorname{tr}}(S_j,\Q)\simeq \HH^2_{\operatorname{tr}}(S_k,\Q)$,
			hence $\operatorname{NS}(S_{j})_{\Q}\simeq \operatorname{NS}(S_{k})_{\Q}$, as
			$\Q$-quadratic spaces.
		\end{itemize} Moreover, such a family can be chosen to consist of K3 surfaces
		with transcendental lattice being any prescribed even lattice with square
		discriminant and of signature $(2, 2), (2, 4), (2, 6)$ or $(2, 8)$.
	\end{thm}

	Furthermore, we will show in Proposition \ref{prop:IsoChowRing3} that, assuming
	Conjecture \ref{conj:MGTT}, the K3 surfaces in such a family all have
	isomorphic
	Chow motives and isomorphic Chow rings, but their Chow
	motives are pairwise non-isomorphic as Frobenius algebra objects. This gives
	evidence that one cannot characterize the isogeny class of a complex K3 surface
	with its Chow ring.
	
	\begin{rmk}\label{R:BSV}
		As mentioned to us by Chiara Camere, examples of a pair of
		non-isogenous K3 surfaces with isomorphic rational Hodge structures were
		constructed geometrically, via the so-called \emph{Inose
			isogenies}\footnote{Note that an Inose isogeny is not an isogeny in our
			sense.}
		\cite{Inose}, by Boissier--Sarti--Veniani \cite{BSV}. Precisely, let $f$ be a
		symplectic automorphism of prime order $p$ of a K3 surface $S$ and let $S'$ be
		the minimal resolution of the quotient $S/\langle f \rangle$\,; the surface
		$S'$
		is a K3 surface and, by definition, the rational map $S\dashrightarrow S'$ is
		a
		degree-$p$ Inose isogeny between these two K3 surfaces. On the one hand, we
		have
		isomorphisms of Hodge structures $\HH^{2}_{\operatorname{tr}}(S,
		\Q)=\HH^{2}_{\operatorname{tr}}(S,
		\Q)^{f}\simeq\HH^{2}_{\operatorname{tr}}(S/\langle f \rangle, \Q)\simeq
		\HH^{2}_{\operatorname{tr}}(S', \Q)$. On the other hand, \cite[Theorem 1.1 and
		Corollary 1.2]{BSV} provide many situations where
		$\HH^{2}_{\operatorname{tr}}(S, \Q)$ and $\HH^{2}_{\operatorname{tr}}(S', \Q)$
		are not isometric and consequently $\HH^{2}(S, \Q)$ and $\HH^{2}(S', \Q)$ are
		not Hodge isometric.
		Note that this approach only produces finitely many non-isogenous K3 surfaces
		with isomorphic Hodge structures. The aim of this appendix is to show, via the
		surjectivity of the period map, that in fact one can produce an
		\emph{infinite}
		family of such pairwise non-isogenous K3 surfaces.
	\end{rmk}

	\subsection{Hodge isomorphic \emph{vs.} Hodge algebra isomorphic}
	
	All the examples of non-isogenous K3 surfaces that we will consider will
	consist of K3 surfaces $S$ and $S'$ whose respective transcendental lattices
	$T$
	and $T'$ become Hodge isometric after some twist, \emph{i.e.} such that $T
	\simeq T'(m)$ as Hodge lattices for some integer $m$. The aim of this paragraph
	is to show that the cohomology algebra of any two such K3 surfaces are Hodge
	algebra isomorphic\,; \emph{cf.} Lemma \ref{lem:algiso} below. This will reduce
	the proofs of Theorems~\ref{thm:non-isogenousK3}
	and~\ref{P:bothisonotisogenous}
	to showing that the K3 surfaces in the families have Hodge isomorphic
	transcendental cohomology groups.\medskip
	
	First we state a general lemma based on the classical classification of
	quadratic forms over $\Q$. This lemma will also be used in the proof of
	Theorem~\ref{P:bothisonotisogenous}.

	\begin{lem}\label{lemma:QuadForm}
		Let $Q$ be a non-degenerate $\Q$-quadratic form of even rank.
		\begin{itemize}
			\item  If $\operatorname{disc}(Q) = 1$ and $\rk (Q) \equiv 0 \ \text{mod}\
			4$, or if $\operatorname{disc}(Q) = -1$ and $\rk (Q) \equiv 2 \ \text{mod}\
			4$,
			then for any $m\in \Q_{>0}$, $Q$ and $Q(m)$ are isometric $\Q$-quadratic
			forms.
			\item  If $\operatorname{disc}(Q) = 1$ and $\rk (Q) \equiv 2 \ \text{mod}\
			4$, or if $\operatorname{disc}(Q) = -1$ and $\rk (Q) \equiv 0 \ \text{mod}\
			4$, 
			then for any $m\in N\Q(i)^{\times }$, $Q$ and $Q(m)$ are isometric
			$\Q$-quadratic forms, where $N\Q(i)^{\times }=\{x^{2}+y^{2} \mid (x, y)\neq
			(0,
			0)\in \Q^{2}\}$ is the norm group of the field extension $\Q(i)/\Q$.
		\end{itemize}
		
	\end{lem}
	\begin{proof}
		Obviously $Q$ and $Q(m)$ have the same rank and signature. As the rank of $Q$
		is even, their discriminants are also the same (that is, only differ by a
		square). By the classification theory of quadratic forms over $\Q$, we only
		need
		to check that for any prime number $\ell$, their $\varepsilon$-invariants are
		equal at all places $\ell$. Assume that $Q$ is equivalent to the diagonal form
		$\langle a_{1}, \ldots, a_{r}\rangle$ where $r$ is the rank of $Q$ and
		$a_{i}\in
		\Q$\,;  its discriminant is thus given by $\operatorname{disc}(Q) =
		\prod_{i=1}^{r}a_{i}$ and we have
		\begin{eqnarray*}
			\varepsilon_{\ell}(Q(m)) & := &\prod_{i<j}(a_{i}m,
			a_{j}m)_{\ell}=\prod_{i<j}(a_{i}, a_{j})_{\ell}\left(\prod_{i=1}^{r}(a_{i},
			m)_{\ell}\right)^{r-1}(m, m)_{\ell}^{r(r-1)/2}\\
			& = &\varepsilon_{\ell}(Q) \big((\operatorname{disc}(Q)^{r-1}
			(-1)^{r(r-1)/2},m\big)_\ell.
		\end{eqnarray*}
		where $(a, b)_{\ell}$ is the Hilbert symbol of $a, b$ for the local field
		$\Q_{\ell}$ and the last equality uses the identity $(m, m)_{\ell}=(m,
		-1)_{\ell}$. One concludes by using the fact that $(m, -1)_{\ell} = 1$ for all
		prime numbers $\ell$ when $m\in N\Q(i)^{\times }$.
	\end{proof}

	\begin{lem}\label{lem:algiso}
		Let $S$ and $S'$ be two complex K3 surfaces.  The following statements are
		equivalent\,:
		\begin{enumerate}[(i)]
			\item There is a Hodge isometry $\HH^{2}_{\operatorname{tr}}(S', \Q) \simeq
			\HH^{2}_{\operatorname{tr}}(S, \Q)(m)$, for some $m\in \Q_{>0}$\,;
			\item There is a Hodge isometry $\HH^2(S',\Q) \simeq \HH^2(S,\Q)(m)$, for
			some
			$m\in \Q_{>0}$\,;
			\item The cohomology rings $\HH^*(S,\Q)$ and $\HH^*(S',\Q)$ are Hodge algebra
			isomorphic.
		\end{enumerate}
	\end{lem}
	\begin{proof}
		$(i) \Rightarrow (ii)$.
		Clearly any choice of linear isomorphism between the orthogonal complements of
		$\HH^{2}_{\operatorname{tr}}(S, \Q)$ and $\HH^{2}_{\operatorname{tr}}(S', \Q)$
		provides a Hodge isomorphism
		$\HH^{2}(S, \Q) \simeq \HH^{2}(S', \Q)$ that extends the Hodge isomorphism
		between $\HH^{2}_{\operatorname{tr}}(S, \Q)$ and
		$\HH^{2}_{\operatorname{tr}}(S', \Q)$. Since the K3
		lattice $\Lambda:=E_{8}(-1)^{\oplus 2}\oplus	U^{\oplus 3}$  and its twist
		$\Lambda(m)$ are $\Q$-isometric for all positive rational numbers $m$
		(\emph{cf.}
		Lemma~\ref{lemma:QuadForm}), the Hodge isometry
		$\HH^{2}_{\operatorname{tr}}(S', \Q) \simeq \HH^{2}_{\operatorname{tr}}(S,
		\Q))(m)$ extends to a Hodge
		isometry $\HH^2(S',\Q) \simeq \HH^2(S,\Q)(m)$ by Witt's theorem, where the
		twist $(m)$ refers to
		the quadratic form.
		
		$(ii) \Rightarrow (iii)$. Item $(ii)$ means that there is a Hodge class
		$\Gamma
		\in
		\HH^4(S\times S',\Q)$ making the diagram 	
		\begin{equation}\label{diag:degree}
		\xymatrix{\HH^2(S,\Q)\otimes \HH^2(S,\Q) \ar[rr]^{\quad \cup}
			\ar[d]^{(\Gamma\otimes
				\Gamma)_*} && \HH^4(S,\Q) 
			\ar[rr]^{\deg} && \Q \ar[d]^{\cdot m}\\
			\HH^2(S',\Q)\otimes \HH^2(S',\Q) \ar[rr]^{\quad \cup} && \HH^4(S',\Q)
			\ar[rr]^{\deg}
			&&
			\Q
		}
		\end{equation} 
		commute. By
		imposing the $(4,0)$-  and $(0,4)$-K\"unneth components of $\Gamma$ to be
		$[\pt]
		\times [S']$ and $m\, [S]\times [\pt]$, respectively, we obtain that $\Gamma_*
		: 
		\HH^*(S,\Q) \to \HH^*(S',\Q)$ is an isomorphism of algebras.
		
		$(iii) \Rightarrow (ii)$. Let $\Gamma : \HH^*(S,\Q) \to \HH^*(S',\Q)$ be a
		Hodge
		algebra isomorphism. Then its restriction to $\HH^2(S,\Q)$ induces the
		commutative diagram \eqref{diag:degree}, where $m$ is the rational number such
		that $\Gamma_*[\pt] = m\, [\pt]$.
		
		$(ii)\Rightarrow (i)$.  As any Hodge isomorphism must preserve the
		transcendental part, we see that the restriction of a Hodge isometry as in
		$(ii)$ gives a Hodge isometry $\HH^{2}_{\tr}(S, \Q)\to \HH^{2}_{\tr}(S', \Q)$.
	\end{proof}

	\newpage
	\subsection{Proof of Theorems~\ref{thm:non-isogenousK3}
		and~\ref{P:bothisonotisogenous}} \label{SS:1}

	\subsubsection{Hodge algebra isomorphic but non-isometric transcendental
		cohomology}
	This lattice-theoretic approach to show the existence of non-isogenous K3
	surfaces with Hodge isomorphic second cohomology group was communicated to us
	by
	Benjamin Bakker, who
	attributes it to Huybrechts. 
	Let $S$ be a projective K3 surface with Picard
	number $\rho$. Denote its transcendental lattice by
	$T:=\HH^{2}_{\operatorname{tr}}(S, \Z)$\,; it is an even lattice of signature
	$(2, 20-\rho)$. 
	By Nikulin's embedding theorem \cite[Theorem 1.14.4]{Nikulin} (see also
	\cite[Theorem \textbf{14}.1.12, Corollary \textbf{14}.3.5]{HuybrechtsBook}),
	when $12 \leq \rho\leq 20$, for any integer $m>0$, the lattice $T(m)$ admits a
	primitive embedding into the K3 lattice $\Lambda:=E_{8}(-1)^{\oplus 2}\oplus
	U^{\oplus 3}$, unique up to $O(\Lambda)$. Now consider the (new) Hodge
	structure
	on $\Lambda$ given by declaring that the Hodge structure on $T(m)$ is the same
	as the one on $T$ and $T(m)^{\perp}$ is of type $(1, 1)$. By the surjectivity
	of
	the period map, there exists a K3 surface $S_{m}$, such that there is a Hodge
	isometry $T(m)\simeq \HH^{2}_{\operatorname{tr}}(S_{m}, \Z)$. In particular,
	for
	all $m>0$, the Hodge structures $\HH^{2}(S_{m}, \Q)$ are all isomorphic to
	$\HH^{2}(S, \Q)$ and in fact, for all $m>0$, the cohomology algebras
	$\HH^*(S_m,\Q)$ are Hodge algebra isomorphic due to Lemma~\ref{lem:algiso}.
	
	\begin{proof}[Proof of Theorem~\ref{thm:non-isogenousK3}]
		We now take $\rho=20$ and let $S$ be the Fermat quartic surface\,; its
		transcendental lattice $T$ is isomorphic to $\Z(8)\oplus \Z(8)$. In order to
		prove Theorem \ref{thm:non-isogenousK3}, it is enough to construct an infinite
		sequence of positive integers $\{m_{j}\}_{j=1}^{\infty}$, such that the
		lattices
		$T(m_{j})\simeq \Z(8m_{j}) \oplus \Z(8m_{j})$ are pairwise non-isometric over
		$\Q$. However, it is easy to see that for $m, m'\in \Z_{>0}$, the two
		$\Q$-quadratic forms $ \Q(8m) \oplus \Q(8m)$ and  $ \Q(8m') \oplus \Q(8m')$
		are
		isometric if and only if $mm'$ belongs to $\{x^{2}+y^{2}\mid x, y\in \Z\}$,
		which in turn is equivalent to the condition that any prime factor of $mm'$
		congruent to 3 modulo 4 has even exponent, by the theorem of the sum of two
		squares. Therefore a desired sequence is easy to construct, for example, one
		can
		take
		$m_{j}$ to be the $j$-th prime number
		congruent to 3 modulo 4.
	\end{proof}

	\subsubsection{Hodge algebra isomorphic and isometric but non-Hodge isometric
		transcendental cohomology}
	In the previous example, the K3 surfaces $S_{m_{j}}$ all have isomorphic
	$\HH^{2}(-, \Q)$ as rational Hodge structures and their isogeny classes are
	distinguished from each other by the $\Q$-quadratic forms on
	$\HH^{2}_{\operatorname{tr}}(-, \Q)$. We would like to go further and produce
	non-isogenous K3 surfaces with $\HH^{2}_{\operatorname{tr}}(-, \Q)$ both
	isomorphic as Hodge structures and isometric as $\Q$-quadratic forms. Let us
	first remark that no such examples of K3 surfaces exist in the case of maximal
	Picard number $\rho=20$. Indeed, we have the following elementary result\,:
	
	\begin{lem}\label{lemma:rho=20}
		Let $S, S'$ be two K3 surfaces with maximal Picard number $\rho=20$.
		\begin{itemize}
			\item  If $\HH^{2}_{\operatorname{tr}}(S, \Z)$ and
			$\HH^{2}_{\operatorname{tr}}(S', \Z)$ are isometric lattices, then $S$ and
			$S'$
			are isomorphic. 
			\item If $\HH^{2}_{\operatorname{tr}}(S, \Q)$ and
			$\HH^{2}_{\operatorname{tr}}(S', \Q)$ are isometric $\Q$-quadratic forms,
			then
			$S$ and $S'$ are isogenous. 
			\item  If $\HH^{2}_{\operatorname{tr}}(S, \Q)$ and
			$\HH^{2}_{\operatorname{tr}}(S', \Q)$ are Hodge isomorphic, then
			$\HH^*(S,\Q)$ and $\HH^*(S',\Q)$ are Hodge algebra isomorphic.
		\end{itemize} 
	\end{lem}
	\begin{proof}
		We only prove the first two points\,; the third is left to the reader.
		Let $T$ be the quadratic space underlying their transcendental cohomologies.
		Due
		to the Hodge--Riemann bilinear relations, $T$ is positive definite. Choose an
		orthogonal basis $\{e_{1}, e_{2}\}$ of $T$ and let $d_{i}:= (e_{i}, e_{i})\in
		\Q_{>0}$. One observes that there are only two isotropic directions in
		$T\otimes
		\C$, namely $\sqrt{d_{1}}e_{1}\pm i\sqrt{d_{2}}e_{2}$, hence only two possible
		Hodge structures of K3 type on~$T$. However, these two Hodge structures are
		Hodge isometric via the $\Q$-linear transformation $e_{1}\mapsto e_{1}$;
		$e_{2}\mapsto -e_{2}$.
	\end{proof}
	
	For K3 surfaces with $\rho=12, 14, 16, 18$, there are indeed examples of
	non-isogenous K3 surfaces with $\HH_{\operatorname{tr}}^{2}(-, \Q)$ both
	isometric 
	and isomorphic as rational Hodge structures, as is stated in Theorem
	\ref{P:bothisonotisogenous}.

	\begin{proof}[Proof of Theorem \ref{P:bothisonotisogenous}]
		Given any even lattice $T$ of signature $(2, 2), (2, 4), (2, 6)$ or $(2, 8)$
		whose discriminant is a square, by Lemma \ref{lemma:QuadForm}, the
		$\Q$-quadratic forms $T(m)\otimes \Q$ and $T\otimes \Q$ are isometric for any 
		integer $m\in \Z_{>0}$ which is the sum of two squares. On the other hand, a
		generic choice of an isotropic element $\sigma\in T\otimes \C$ gives rise to
		an
		irreducible $\Q$-Hodge structure on $T$, hence on all its twists $T(m)$, with
		minimal endomorphism algebra\,; i.e. $\End_{HS}(T_{\Q})\simeq \Q$. So for all
		$m\in \Z_{>0}$ which is the sum of two squares, the twists $T(m)\otimes \Q$
		are
		all Hodge isomorphic and isometric. However, since for any such integers $m$
		and
		$m'$, $\Hom_{HS}\left(T(m)\otimes \Q, T(m')\otimes \Q\right)=\Q$, we see that
		$T(m)\otimes \Q$ and $T(m')\otimes \Q$ are Hodge isometric if and only if
		$mm'$
		is a square.
		
		In order to realize the twists $ T(m)$ as transcendental lattices of K3
		surfaces, we use Nikulin's embedding theorem \cite[Theorem 1.14.4]{Nikulin} to
		get a primitive embedding of $T(m)$ into the K3 lattice $\Lambda$.
		For each $m\in \Z_{>0}$, we can therefore construct a Hodge structure on
		$\Lambda$ by declaring that $T(m)$ carries the Hodge structure on $T$ and
		$T(m)^{\perp}$ is of type $(1, 1)$. By the surjectivity of the period map, for
		any $m \in \Z_{>0}$, there exists a K3 surface $S_{m}$ with
		$\HH^{2}_{\operatorname{tr}}(S_{m}, \Q)$ Hodge isometric to $T(m)$.
		
		Now, thanks to Lemma~\ref{lem:algiso}, it remains to construct an infinite
		sequence of positive integers $\{m_{j}\}_{j=1}^{\infty}$ which are sums of two
		squares, such that the product of any two different terms is not a square.
		This
		is easily achieved\,: for example, one can take  $m_{j}$ to be the $j$-th
		prime
		number congruent to 1 modulo~4.
	\end{proof}

	\subsection{Consequences on motives}  \label{SS:motivic}
	
	\subsubsection{The general expectations}
	The following conjecture is a combination of the Hodge conjecture and of the
	conservativity conjecture (which itself is a consequence of the
	Kimura--O'Sullivan finite-dimensionality conjecture or of the Bloch--Beilinson
	conjectures).
	
	\begin{conj}\label{conj:MGTT}
		Two smooth projective varieties $X$ and $Y$ have
		isomorphic Chow motives if and only if their rational cohomologies are
		isomorphic as Hodge structures\,: 
		\begin{equation*}\label{eqn:MotTorelli}
		\h(X)\simeq \h(Y) \ \text{as Chow motives} \  \Longleftrightarrow \HH^{*}(X,
		\Q)\simeq \HH^{*}(Y, \Q) \ \text{as graded Hodge structures}.
		\end{equation*}
	\end{conj}
	
	The implication $\Rightarrow$ holds unconditionally and is simply attained by
	applying the Betti realization functor. Regarding the implication $\Leftarrow$,
	the Hodge conjecture predicts that an isomorphism $\HH^{*}(X, \Q)\simeq
	\HH^{*}(Y, \Q)$ of graded Hodge structures and its inverse are induced by the
	action of a correspondence. Hence the homological motives of $X$ and $Y$ are
	isomorphic. By conservativity, such an isomorphism lifts to rational
	equivalence, \emph{i.e.} lifts to an isomorphism between the Chow motives of
	$X$
	and $Y$.
	\medskip

	Obviously, if $	\h(X)\simeq \h(Y) $ as (Frobenius) algebra objects in the
	category of Chow motives, then by realization $\HH^{*}(X, \Q)$ and $\HH^{*}(Y,
	\Q) $ are Hodge (Frobenius) algebra isomorphic. 
	We would like to discuss to which extent the converse statement could be true.
	In general, this is not the case\,: consider for instance a complex K3 surface
	$S$\,; then the blow-up $S_1$ of $S$ at a point lying on a rational curve and
	the blow-up $S_2$ of $S$ at a very general point have Hodge isomorphic
	cohomology Frobenius algebras, but due to the Beauville--Voisin Theorem
	\ref{thm:bv} we have $$1 = \rk \left(\CH^1(S_1) \otimes \CH^1(S_1) \to
	\CH^2(S_1)\right) \neq \rk \left(\CH^1(S_2) \otimes \CH^1(S_2) \to
	\CH^2(S_2)\right) =2,$$ in particular their Chow motives are not isomorphic as
	algebra objects.\medskip
	
	However, in the case of hyper-K\"ahler varieties, one
	can expect\,:
	
	\begin{conj}\label{conj:MGTT2}
		Two smooth projective hyper-K\"ahler varieties $X$ and $Y$ have
		isomorphic Chow motives as (Frobenius) algebra objects if and only if their
		rational cohomology rings
		are Hodge (Frobenius) algebra	isomorphic \,: 
		\begin{align*}
		& \h(X)\simeq \h(Y) \ \text{as (Frobenius) algebra objects} \\
		&\Longleftrightarrow \HH^{*}(X, \Q)\ \text{and} \  \HH^{*}(Y, \Q) \ \text{are
			Hodge (Frobenius) algebra isomorphic}.
		\end{align*}
	\end{conj}
	
	We note that Corollary~\ref{cor:torelli} establishes this conjecture in the
	case
	of K3 surfaces and Frobenius algebra structures. In general,
	Conjecture~\ref{conj:MGTT2} is implied by the combination of the Hodge
	conjecture and of the ``distinguished marking conjecture'' for hyper-K\"ahler
	varieties \cite[Conjecture~2]{fv}.
	
	\begin{prop}\label{prop:conj}
		Let $X$ and $Y$ be hyper-K\"ahler varieties of same dimension $d$. Assume\,:
		\begin{itemize}
			\item The Hodge conjecture in codimension $d$ for $X\times Y$\,;
			\item  $X$ and $Y$ satisfy the ``distinguished marking conjecture''
			\cite[Conjecture~2]{fv}.
		\end{itemize}
		Then Conjecture~\ref{conj:MGTT2} holds for $X$ and $Y$.
	\end{prop}
	\begin{proof}
		By \cite[\S 3]{fv}, the distinguished marking conjecture for $X$ and $Y$ 
		provides for all non-negative integers $n$ and $m$ a section to the graded
		algebra epimorphism $\CH^*(X^n\times Y^m) \to \overline{\CH}^*(X^n\times Y^m)$
		in such a way that these are compatible with push-forwards and pull-backs along
		projections. In addition, the images of the sections corresponding to
		$\CH^*(X^2) \to \overline{\CH}^*(X^2)$ and $\CH^*(Y^2) \to
		\overline{\CH}^*(Y^2)$ contain the diagonals $\Delta_X$ and $\Delta_Y$,
		respectively. Here, $\overline{\CH}^*(-)$ denotes the Chow ring modulo numerical
		equivalence. In fact, since numerical and homological equivalence agree for
		abelian varieties, the same holds for $X$ and $Y$ (\emph{via} their markings). 
		
		As before, the Hodge conjecture predicts that a Hodge isomorphism $\HH^{*}(X,
		\Q)\simeq \HH^{*}(Y, \Q)$ and its inverse are induced by the action of an
		algebraic correspondence. We fix the isomorphism $\phi: \h(X)
		\stackrel{\sim}{\longrightarrow} \h(Y)$ to be the correspondence that is the
		image of the Hodge class under the section to $\CH^*(X\times Y) \to
		\overline{\CH}^*(X\times Y)$ inducing the Hodge isomorphism of (Frobenius)
		algebras $\HH^{*}(X, \Q) \stackrel{\sim}{\longrightarrow} \HH^{*}(Y, \Q)$. Since
		the (Frobenius) algebra structure on the motives of varieties is simply
		described in terms of the rational equivalence class of the diagonal and of the
		small diagonal, the isomorphism $\phi$ provides, thanks to the compatibilities
		of the sections on the product of various powers of $X$ and $Y$, a morphism
		compatible with the (Frobenius) algebra structures.
	\end{proof}

	Although we do not know how to establish Conjecture~\ref{conj:MGTT2} in general
	for K3 surfaces and algebra structures, we can still say something for K3
	surfaces with a \emph{Shioda--Inose structure}.
	Recall that a Shioda--Inose structure on a K3
	surface $S$ consists of a Nikulin
	involution (that is, a symplectic involution) with rational quotient map $\pi :
	S\dashrightarrow Y$ such that $Y$ is a Kummer surface and $\pi_*$ induces a
	Hodge isometry $T_S(2) \simeq T_Y$, where $T_S$ and $T_Y$ denote the
	transcendental lattices of $S$ and $Y$. If $S$ admits a Shioda--Inose
	structure,
	let $f: A\to Y$ be the quotient morphism from the complex abelian surface whose
	Kummer surface is $Y$. By \cite[\S
	6]{morrison},  there is a Hodge isometry of transcendental lattices $T_S \simeq
	T_A$, and $f^*\pi_*$ induces an isomorphism $\HH_{\mathrm{tr}}^2(S,\Q)
	\stackrel{\sim}{\longrightarrow} \HH^2_{\mathrm{tr}}(A,\Q)$ with inverse
	$\frac{1}{2} f_*\pi^*$.
	\medskip

	\begin{prop}
		\label{prop:motivicglobalTorelli} 
		Let $S$ and $S'$ be two K3 surfaces with a Shioda--Inose structure
		(\emph{e.g.}
		with
		Picard rank $\geq 19$, 	\cite[Corollary~6.4]{morrison}). The following
		conditions are equivalent. 
		\begin{enumerate}[(i)]
			\item  $\HH^*(S,\Q)$ and $\HH^*(S',\Q)$  are Hodge algebra isomorphic.
			\item$\h(S) \simeq \h(S')$ as algebra objects in the category of rational
			Chow
			motives.
		\end{enumerate}
	\end{prop}
	
	\begin{proof}
		Let $S$ and $S'$ be two K3 surfaces with a Shioda--Inose structure. In a
		similar vein to Proposition~\ref{prop:conj},
		the proposition is a combination of the validity
		of the Hodge conjecture for $S\times S'$, together with the fact
		\cite[Proposition~5.12]{fv} that $S$ and $S'$
		satisfy the distinguished marking conjecture of \cite[Conjecture~2]{fv}. 
		The fact that the Hodge conjecture holds for $S\times S'$ reduces, \emph{via}
		the correspondence-induced isomorphism $\HH_{\mathrm{tr}}^2(S,\Q)
		\stackrel{\sim}{\longrightarrow} \HH^2_{\mathrm{tr}}(A,\Q)$ described above,
		to
		the fact that the Hodge conjecture holds for the product of any two
		abelian surfaces. The latter is proven in \cite{rm}.
	\end{proof}

	\subsubsection{Non-isogenous K3 surfaces with Chow motives isomorphic as
		algebra
		objects}
	By combining Theorem~\ref{thm:non-isogenousK3}  with the fact
	\cite[Corollary~6.4]{morrison} that K3 surfaces of maximal Picard rank admit a
	Shioda--Inose structure, we can establish\,:
	
	\begin{thm}\label{thm:IsomChowMot}
		There exists an infinite family of K3 surfaces such that
		\begin{itemize}
			\item they are pairwise non-isogenous\,;
			\item their Chow motives are pairwise non-isomorphic as Frobenius algebra
			objects\,;
			\item their Chow motives are all isomorphic as algebra objects.
		\end{itemize}
	\end{thm}
	\begin{proof}
		Let $\{S_i\}_{i \in \Z_{>0}}$  be a family of pairwise non-isogenous K3
		surfaces of
		maximal Picard rank such that $\HH^*(S_j,\Q)$ and $\HH^*(S_k,\Q)$ are
		Hodge algebra isomorphic for all $j,k \in \Z$. Such a family of K3
		surfaces
		exist thanks to Theorem~\ref{thm:non-isogenousK3}.
		By Corollary \ref{cor:torelli}  the Chow motives of these K3 surfaces are
		pairwise non-isomorphic as Frobenius algebra objects. The fact that
		the Chow motives
		of any two surfaces  in the family are isomorphic as algebra objects is
		Proposition~\ref{prop:motivicglobalTorelli}.
	\end{proof}

	Finally, the following proposition gives evidence that the notion of
	``isogeny'' for K3 surfaces is strictly more restrictive than the notion of
	``isomorphic Chow
	rings'' (see Remark~\ref{rmk:ChowRingIso})\,:

	\begin{prop} \label{prop:IsoChowRing3}
		Assume that Conjecture~\ref{conj:MGTT} holds for K3 surfaces.
		Then there exists an infinite family $\{S_i\}_{i \in \Z_{>0}}$  of pairwise
		non-isogenous K3 surfaces with the property
		that, for all $j,k \in \Z_{>0}$,  there exists an isomorphism $\h(S_j)
		\stackrel{\sim}{\longrightarrow}\h(S_k)$ of Chow motives inducing a ring
		isomorphism $\CH^*(S_j) \stackrel{\sim}{\longrightarrow} \CH^*(S_k)$ such that
		the distinguished class $o_{S_j}$ is mapped to the distinguished class
		$o_{S_k}$.
		
		Moreover, such a family can be chosen to consist of K3 surfaces
		with transcendental lattice being any prescribed even lattice with square
		discriminant and of signature $(2, 2), (2, 4), (2, 6)$ or $(2, 8)$.
	\end{prop}

	\begin{proof}
		We consider the infinite family constructed in Theorem
		\ref{P:bothisonotisogenous}. For any $j\neq k$, $S_{j}$ and $S_{k}$ are not
		isogenous. On the other hand, Conjecture \ref{conj:MGTT} implies that the Chow
		motives of $S_{j}$ and $S_{k}$ are isomorphic\,; in particular, by the same
		weight argument
		as in \S \ref{subsec:weight}, there exists an isomorphism between their
		transcendental motives\,: $$\Gamma_{\operatorname{tr}}:
		\h^{2}_{\operatorname{tr}}(S_{j})\simeq \h^{2}_{\operatorname{tr}}(S_{k}).$$
		As
		$\operatorname{NS}(S_{j})_{\Q}$ and $\operatorname{NS}(S_{k})_{\Q}$ are
		isometric by construction, there is an isomorphism between the algebraic part
		of
		their weight-2 motives $\Gamma^{2}_{\operatorname{alg}}: 
		\h^{2}_{\operatorname{alg}}(S_{j})\to\h^{2}_{\operatorname{alg}}(S_{k})$ which
		induces the isometry between the N\'eron--Severi spaces. Combining them
		together, $\Gamma:=o_{S_{j}}\times
		S_{k}+\Gamma_{\operatorname{alg}}^{2}+\Gamma_{\operatorname{tr}}+S_{j}\times
		o_{S_{k}}$ yields an isomorphism between their Chow motives\,:
		$$\Gamma: \h(S_{j})=\h^{0}(S_{j})\oplus
		\h^{2}_{\operatorname{alg}}(S_{j})\oplus
		\h^{2}_{\operatorname{tr}}(S_{j})\oplus
		\h^{4}(S_{j})\stackrel{\sim}{\longrightarrow} \h(S_{k})=\h^{0}(S_{k})\oplus
		\h^{2}_{\operatorname{alg}}(S_{k})\oplus
		\h^{2}_{\operatorname{tr}}(S_{k})\oplus
		\h^{4}(S_{k}),$$
		with the extra property that  it induces an isometry between the
		$\Q$-quadratic spaces $\CH^{1}(S_{j})$ and $\CH^{1}(S_{k})$. Now as in Remark
		\ref{rmk:ChowRingIso}, according to the Beauville--Voisin theorem \cite{bv},
		the
		image of the intersection product $\CH^{1}(S_{j})\otimes
		\CH^{1}(S_{j})\to\CH^{2}(S_{j})$ is 1-dimensional and similarly for~$S_{k}$.
		This implies that $\Gamma$ induces an isomorphism  of Chow rings with
		$\Gamma_*o_{S_j} = o_{S_k}$.
	\end{proof}

	\begin{rmk}
		In Proposition~\ref{prop:IsoChowRing3}, if one assumes
		Conjecture~\ref{conj:MGTT2} for K3 surfaces instead of
		Conjecture~\ref{conj:MGTT}, then the K3 surfaces of the family have isomorphic
		Chow motives as algebra objects.
	\end{rmk}

	\subsection{An elementary proof of Theorem~\ref{thm:non-isogenousK3}} \label{SS:2}
	We provide an alternate proof of Theorem~\ref{thm:non-isogenousK3} which is elementary in the sense that it does not rely on Nikulin's deep lattice-theoretic result. \medskip
	
	Given a K3 surface $S$ of Picard rank 20, the subspace $\HH^{2,0}(S)\oplus
	\HH^{0,2}(S) \subset
	\HH^2(S,\Q)\otimes \C$ is defined over $\Q$, and consequently, up to
	multiplying
	by a real scalar, we can fix a
	generator $\sigma$ of the 1-dimensional space $\HH^{2,0}(S)$ so that
	$\sigma+\bar\sigma$ is rational, i.e., lies in $\HH^{2}(S, \Q)$. 
	
	Let $(-,-)$ denote the intersection pairing and set $v:=(\sigma, \bar\sigma) >
	0$. Since $(\sigma+\bar\sigma, \sigma+\bar\sigma)=2(\sigma, \bar\sigma)=2v$, we
	see that $v\in \Q_{>0}$. Since the real element $i(\sigma-\bar\sigma)$ spans
	the
	orthogonal complement of the rational line spanned by $\sigma + \bar{\sigma}$,
	the real line spanned by $i(\sigma-\bar\sigma)$ is also defined over $\Q$.
	Since
	$i(\sigma-\bar\sigma)$ has norm
	$(i(\sigma-\bar\sigma), i(\sigma-\bar\sigma))=2(\sigma, \bar\sigma)=2v$, there
	exists a rational number $\alpha>0$ such that
	$i\sqrt{\alpha}(\sigma-\bar\sigma)\in \HH^{2}(S, \Q)$. 
	
	Let $\{D_{1}, \ldots, D_{\rho}\}$ be an orthogonal basis of the N\'eron--Severi
	space $\operatorname{NS}(S)_{\Q}$. Write $(D_{j},D_j)=2d_{j}$ with $d_{j}\in
	\Q$. We then have the following $\Q$-basis of $\HH^{2}(S, \Q)$\,:
	$$\{\sigma+\bar\sigma, i\sqrt{\alpha}(\sigma-\bar\sigma), D_{1}, \ldots,
	D_{\rho}\}.$$
	
	Let $V$ be the lattice $\HH^{2}(S, \Z)$ equipped with the intersection pairing.
	We aim, via the (surjective) period map, to construct new K3 surfaces by
	constructing new Hodge structures on the lattice $V$.
	To this end, define a $\Q$-linear map $\phi$ from the transcendental part
	$\HH^{2}_{\operatorname{tr}}(S, \Q)$ to $V_{\Q}$ by the following formula
	\begin{equation} \label{eqn:formula}
	\begin{cases}
	\phi(\sigma)=a\sigma+\bar b\bar\sigma+\sum_{j=1}^{\rho}c_{j}D_{j}\\
	\phi(\bar\sigma)=b\sigma+\bar a\bar\sigma+\sum_{j=1}^{\rho}\bar c_{j}D_{j}
	\end{cases}
	\end{equation}
	with $a, b, c_{j}\in \C$.
	The condition that the image of $\phi$ lies in
	$V_{\Q}$ is equivalent to saying that 
	\begin{equation}\label{eq:Qalpha}
	a, b, c_{j}\in \Q(\sqrt{-\alpha}),
	\end{equation}
	where $\Q(\sqrt{-\alpha})$ is the imaginary quadratic extension of $\Q$ with
	basis 1 and $i\sqrt{\alpha}$.
	
	Consider the following subspaces of $V_{\C}$\,:
	$$V^{2,0}:=\C\cdot\phi(\sigma), \quad V^{0,2}:=\C\cdot \phi(\bar\sigma), \quad
	V^{1,1}:=(V^{0,2}+V^{2,0})^{\perp}.$$
	By the global Torelli theorem \cite{PSS}, they define a Hodge structure of a K3
	surface $S'$ if and only if 
	\begin{equation*}
	(\phi(\sigma), \phi(\sigma))=0 \text{ and } (\phi(\sigma), \phi(\bar\sigma))>0,
	\end{equation*}
	or equivalently, the following numerical conditions are satisfied\,: 
	\footnote{\label{fn:ab} We observe that these conditions imply that $a$ and $b$
		cannot be both zero. Indeed, by the Hodge index theorem, one of the $d_i$'s is
		positive, say $d_1$, while all the others are negative. The triangular
		inequality applied to $c_1^2 d_1 = \sum_{j=2}^{\rho}c_{j}^{2}(-d_{j})$ yields
		$|c_{1}|^{2}d_{1} \leq \sum_{j=2}^{\rho}|c_{j}|^{2}(-d_{j})$, i.e.,
		$\sum_{j=1}^{\rho}|c_{j}|^{2}d_{j}\leq 0$. }
	\begin{equation}\label{eqn:constraints}
	\begin{cases}
	va\bar{b}+\sum_{j=1}^{\rho}c_{j}^{2}d_{j}=0\\
	v(|a|^{2}+|b|^{2})+2\sum_{j=1}^{\rho}|c_{j}|^{2}d_{j}>0.
	\end{cases}
	\end{equation}
	By construction, there is an isomorphism between the $\Q$-Hodge structures of
	the two K3 surfaces $S$ and $S'$\,:
	$$\widetilde \phi: \HH^{2}(S, \Q)\stackrel{\sim}{\longrightarrow} \HH^{2}(S',
	\Q),$$
	where $\widetilde\phi$ is given on $\HH^{2}_{\operatorname{tr}}(S, \Q)$ by
	$\phi$ and on $\operatorname{NS}(S)_{\Q}$ by any isomorphism to
	$\operatorname{NS}(S')_{\Q}$. 
	
	Note that there is \emph{a priori} no reason for $\widetilde \phi$ to be an
	isometry. Our goal is actually to provide examples where no Hodge isometry
	exists. If there exists a Hodge isometry 
	$$\psi: \HH^{2}(S, \Q)\to \HH^{2}(S', \Q),$$
	then as $\HH^{2,0}$ is 1-dimensional, there is a $\lambda\in \C^{\times}$, such
	that 
	\begin{equation*}
	\psi(\sigma)=\frac{1}{\lambda}\phi(\sigma)=\frac{1}{\lambda}\left(a\sigma+\bar
	b\bar\sigma+\sum_{j=1}^{\rho}c_{j}D_{j}\right).
	\end{equation*}
	Since $\psi$ respects the $\Q$-structures, we find that $\frac{a}{\lambda},
	\frac{b}{\lambda}, \frac{c_{j}}{\lambda} \in \Q(\sqrt{-\alpha})$, hence from
	\eqref{eq:Qalpha} that $\lambda\in \Q(\sqrt{-\alpha})^\times$.
	The condition that $\psi$ is an isometry implies in particular that
	\begin{equation*}
	(\psi(\sigma), \psi(\bar\sigma))=(\sigma, \bar\sigma),
	\end{equation*}
	that is, 
	\begin{equation*}
	|\lambda|^{2}=|a|^{2}+|b|^{2}+2\sum_{j=1}^{\rho}|c_{j}|^{2}\frac{d_{j}}{v}.
	\end{equation*}
	
	To summarize, any solution $a, b, c_{j}\in \Q(\sqrt{-\alpha})$ of
	\eqref{eqn:constraints} gives rise to a K3 surface $S'$ with $\HH^{2}(S, \Q)$
	isomorphic to $\HH^{2}(S, \Q)$ as $\Q$-Hodge structures and with $\HH^{2}(S,
	\Q)\otimes \langle m \rangle$
	isometric to $\HH^{2}(S', \Q)$ with $m := (\phi(\sigma),\phi(\bar{\sigma})) /
	v$, but it would be
	isogenous to $S$ only if
	\begin{equation}\label{eqn:Isometry}
	|a|^{2}+|b|^{2}+2\sum_{j=1}^{\rho}|c_{j}|^{2}\frac{d_{j}}{v} \in
	N\Q(\sqrt{-\alpha})^\times,
	\end{equation}
	where  $N\Q(\sqrt{-\alpha})^\times=\left\{|z|^{2} \mid z\in
	\Q(\sqrt{-\alpha})^\times\right\}=\left\{x^{2}+\alpha y^{2}\mid (x, y) \neq
	(0,0) \in \Q^2\right\}$ is the norm group of the extension
	$\Q(\sqrt{-\alpha})/\Q$.
	By the same argument, for two such K3 surfaces $S'$ and $S''$, corresponding to
	solutions $(a, b, c_{j})$ and $(a', b', c_{j}')$ of \eqref{eqn:constraints}, if
	they are isogenous, then 
	\begin{equation}\label{eqn:Isometry2}
	\frac{|a|^{2}+|b|^{2}+2\sum_{j=1}^{\rho}|c_{j}|^{2}d_{j}/v
	}{|a'|^{2}+|b'|^{2}+2\sum_{j=1}^{\rho}|c'_{j}|^{2}d_{j}/v }\in
	N\Q(\sqrt{-\alpha})^\times.
	\end{equation}

	\begin{proof}[Alternate proof of Theorem~\ref{thm:non-isogenousK3}]
		By the
		previous discussion (together with Lemma~\ref{lem:algiso}), it is enough to
		provide infinitely many solutions of
		\eqref{eqn:constraints} such that \eqref{eqn:Isometry} does not hold for each
		of
		them and such that \eqref{eqn:Isometry2} does not hold for any two of them.
		
		Let $S=\left(T_{0}^{4}+\cdots+T_{3}^{4}=0\right)\subset \mathbb{P}^{3}$ be the
		Fermat quartic surface. We know (see \cite[Appendix~A]{SD})
		 that $S$ is the
		Kummer surface associated to an abelian surface $A$, which has a degree-two
		isogeny from $E\times E$, where $E\simeq \C/\Z\oplus \Z\cdot i$ is the
		elliptic
		curve with $j$-invariant 1728.
		
		\begin{equation*}
		\xymatrix{
			&&& & S\ar[d]\\
			E\times E\ar@{-->}[urrrr]^{\pi}\ar[rrr]_{\text{degree-2 isogeny}}&& &
			A\ar[r]&
			A/\{\pm 1\}
		}
		\end{equation*}
		
		Consider the basis $\{e_{1}:=1, e_{2}:=i\}$ of $\HH_{1}(E, \Z)$ and let
		$\{e_{1}^{*}, e_{2}^{*}\}$ be the dual basis of $\HH^{1}(E, \Z)$. Denote by
		$z$
		the holomorphic coordinate of $\C$, so that $\HH^{1,0}(E)$ is generated by
		$dz$.
		We see that 
		\begin{equation}\label{eqn:FormInBasis}
		\begin{cases}
		dz=e_{1}^{*}+i e_{2}^{*},\\
		d\bar{z}=e_{1}^{*}-i e_{2}^{*},
		\end{cases}
		\end{equation}
		and that $\int_{E}dz\wedge d\bar{z}=-2i$.
		
		Let $\sigma$ be the generator of $\HH^{2,0}(S)$ such that
		$\pi^{*}(\sigma)=dz_{1}\wedge dz_{2}$ in $\HH^{2,0}(E\times E)$. Using
		\eqref{eqn:FormInBasis}, one checks readily that $\sigma+\bar\sigma$ and
		$i\left(\sigma-\bar\sigma\right)$ belong to the rational lattice $\HH^{2}(S,
		\Q)$ because
		\begin{equation*}
		\begin{cases}
		\pi^{*}(\sigma+\bar\sigma)=dz_{1}\wedge dz_{2}+d\bar{z}_{1}\wedge
		d\bar{z}_{2}=2e_{1,1}^{*}\wedge e_{2,1}^{*}-2e_{1,2}^{*}\wedge e_{2,2}^{*};\\
		\pi^{*}\left(i\sigma-i\bar\sigma\right)=idz_{1}\wedge
		dz_{2}-id\bar{z}_{1}\wedge
		d\bar{z}_{2}=-2e_{1,1}^{*}\wedge e_{2,2}^{*}-2e_{1,2}^{*}\wedge e_{2,1}^{*}.
		\end{cases}
		\end{equation*}
		Hence $\{\sigma+\bar\sigma, i\left(\sigma-\bar\sigma\right)\}$ is a $\Q$-basis
		of $\HH^{2}_{\operatorname{tr}}(S, \Q)$; one can therefore take $\alpha=1$,
		and
		$\Q(\sqrt{-\alpha})$ is simply $\Q(i)$.
		
		On the other hand, $$v=(\sigma,
		\bar\sigma):=\int_{S}\sigma\wedge\bar\sigma=\frac{1}{\deg(\pi)}\int_{E\times
			E}dz_{1}\wedge dz_{2}\wedge d\bar{z}_{1}\wedge
		d\bar{z}_{2}=-\frac{1}{\deg(\pi)}\left(\int_{E}dz\wedge
		d\bar{z}\right)^{2}=1.$$
		
		Now take an orthogonal $\Q$-basis $\{D_{1}, \cdots, D_{\rho}\}$  of
		$\operatorname{NS}(S)_{\Q}$ such that $d_{1}=2$, $d_{2}=-1$\,; this is
		possible
		since the N\'eron--Severi lattice of $S$ is isomorphic to $E_{8}(-1)^{\oplus
			2}\oplus U\oplus \Z(-8)^{\oplus 2}$ (see \cite{SSvL}). We will only use
		solutions of \eqref{eqn:constraints} with $c_{3}=\cdots=c_{\rho}=0$ (and
		remember that $\alpha=v=1$, $d_{1}=2$, $d_{2}=-1$)\,:
		\begin{equation*}
		\begin{cases}
		a\bar{b}+2c_{1}^{2}-c_{2}^{2}=0\\
		|a|^{2}+|b|^{2}+4|c_{1}|^{2}-2|c_{2}|^{2}>0.
		\end{cases}
		\end{equation*}
		For any integer $m>0$, we have a solution
		$$a=2m, \quad b=-\frac{1}{m}(1+i), \quad c_{1}=1, \quad c_{2}=1+i,$$
		where

		$|a|^{2}+|b|^{2}+4|c_{1}|^{2}-2|c_{2}|^{2}=\frac{2}{m^{2}}(2m^{4}+1)=|\frac{1+i}{m}|^{2}(2m^{4}+1)$.
		By the previous discussion, for our purpose, it suffices to produce an
		infinite
		sequence of positive integers $\{m_{j}\}_{j=1}^{\infty}$ such that for any
		$j$,
		$2m_{j}^{4}+1\notin N\Q(i)^\times$ and for any $j\neq k$,
		$\frac{2m_{j}^{4}+1}{2m_{k}^{4}+1}\notin N\Q(i)^\times$, where
		$N\Q(i)=\{x^{2}+y^{2} \mid (x, y)\neq (0,0) \in \Q^2\}.$
		
		We conclude the proof of Theorem \ref{thm:non-isogenousK3}
		by constructing
		such a sequence inductively by means of elementary
		arithmetics. This is the object of Lemma \ref{lem:el} below. 
	\end{proof}

	\begin{lem}\label{lem:el}
		Let $m_{1}:=1$ and for any $j$, $m_{j+1}:=\prod_{l=1}^{j}(2m_{l}^{4}+1)$. 
		Then for any $j$, $2m_{j}^{4}+1\notin N\Q(i)^\times$ and for any $j\neq k$,
		$\frac{2m_{j}^{4}+1}{2m_{k}^{4}+1}\notin N\Q(i)^\times$.
	\end{lem}
	\begin{proof}
		A standard argument of infinite descent shows that a positive integer $n$ is
		not the
		sum
		of squares of two rational numbers, \emph{i.e.}\,not in $N\Q(i)^\times$, if
		and
		only if it is not the sum of squares of two integers. By the theorem of sum of
		two squares, the latter is equivalent to the condition that any prime divisor
		of
		$n$ with odd adic valuation is not congruent to 3 modulo 4.\\
		As the $m_{j}$'s are all odd, $2m_{j}^{4}+1\equiv 3 \mod 4$, hence is not the
		sum of two squares.  On the other hand, by construction, for any $j\neq k$,
		$2m_{j}^{4}+1$ and $2m_{k}^{4}+1$ are coprime to each other, therefore their
		product admits a prime divisor $p$ congruent to 3 modulo 4 with odd adic
		valuation. Hence $\frac{2m_{j}^{4}+1}{2m_{k}^{4}+1}\notin N\Q(i)^\times$.
	\end{proof}

	\begin{rmk}
		The reason for requiring that $c_1$ and $c_2$ are not zero in the proof of
		Theorem~\ref{thm:non-isogenousK3} is the following. Let $\alpha \in \Q_{>0}$,
		and consider a solution to \eqref{eqn:constraints} with $c_2 = \cdots = c_\rho
		=
		0$, that is, a solution to 
		\begin{equation*}\label{eqn:constraints3}
		\begin{cases}
		va\bar{b}+ c^{2}d=0\\
		v(|a|^{2}+|b|^{2})+2|c|^{2}d>0,
		\end{cases}
		\end{equation*}
		with $a,b,c \in \Q(\sqrt{-\alpha})$, $v\in \Q_{>0}$ and $d\in \Z$.
		We observe that $	va\bar{b}+ c^{2}d=0$ implies that $|a||b| = \mp
		|c|^2\frac{d}{v}$ is a non-negative rational number, where the sign depends on
		the sign of $d$. 
		Recall from footnote~\ref{fn:ab} that $a$ and $b$ cannot be both zero\,; if
		one
		of them is zero, we may assume without loss of generality that it is $b$.
		Since $|a|^2$ is a rational number, this implies that $|b| = t |a|$ for some
		rational number $t \in \Q_{\geq 0}$. It is then immediate to check that 
		$$|a|^{2}+|b|^{2}+2|c|^{2}\frac{d}{v} = |a|^{2}+|b|^{2} \pm 2|a||b| =
		|a|^2(1\pm t)^2.$$
		In other words, the equation $	va\bar{b}+ c^{2}d=0$ (with $a$ and $b$ not both
		zero) forces the positivity of $|a|^{2}+|b|^{2}+2|c|^{2}d/v$, but also implies
		that it belongs to $N\Q(\sqrt{-\alpha})^\times$. Therefore, the new K3 surface
		$S'$ obtained, via the global Torelli theorem, by imposing $c_2 = \cdots =
		c_\rho = 0$ in \eqref{eqn:formula} must be isogenous to $S$.
	\end{rmk}

\end{document}